\documentclass[11pt]{amsart}
\usepackage{amsmath,amssymb,amsthm}

\RequirePackage{color}
\RequirePackage[colorlinks,urlcolor=my-blue,linkcolor=my-red,citecolor=my-green]{hyperref}
\definecolor{my-blue}{rgb}{0.0,0.0,0.6}
\definecolor{my-red}{rgb}{0.5,0.0,0.0}
\definecolor{my-green}{rgb}{0.0,0.5,0.0}

\newtheorem{theorem}{\sc Theorem}[section]
\newtheorem{lemma}[theorem]{\sc Lemma}
\newtheorem{proposition}[theorem]{\sc Proposition}

\numberwithin{equation}{section}
\theoremstyle{remark}
\newtheorem{remark}[theorem]{Remark}

\newcommand{\be}{\begin{equation}}
\newcommand{\ee}{\end{equation}}
\newcommand{\nn}{\nonumber}
\providecommand{\abs}[1]{\vert#1\vert}

\def\bE{\mathbb{E}}
\def\bN{\mathbb{N}}
\def\bP{\mathbb{P}}

\def\bR{\mathbb{R}}
\def\bZ{\mathbb{Z}}

\def\om{\omega}

\def\ind{\mathbf{1}}

\def\m1{\mathbf{1}}

\DeclareMathOperator{\Var}{Var}   \DeclareMathOperator{\Cov}{Cov}  
 \def\Vvv{{\rm\mathbb{V}ar}}  \def\Cvv{{\rm\mathbb{C}ov}}

\def\E{\bE}
\def\P{\bP} 
\def\shift{S}  
 
\def\north{{\mathcal N}}  \def\south{{\mathcal S}}  \def\east{{\mathcal E}}  
\def\west{{\mathcal W}}

\def\digamf{\Psi_0} 
\def\trigamf{\Psi_1} 

\def\para{\theta} 
\def\paraa{\lambda} 
\def\uu{Y} 
\def\aaa{a}  \def\bbb{b}    

\newcommand{\fren}{{Z}}

\allowdisplaybreaks[1]
\newcommand{\bea}{\begin{eqnarray}}
\newcommand{\eea}{\end{eqnarray}}
\newcommand{\ben}{\begin{eqnarray*}}
\newcommand{\een}{\end{eqnarray*}}
\newcommand{\eqd}{\stackrel{d}{=}}
\def\tZ{\tilde Z}
\title[Polymer in a Brownian environment]{Bounds for scaling exponents for a 1+1 dimensional directed polymer in a Brownian environment}
\date{\today}
\begin{document}
\author[T.~Sepp\"al\"ainen]{Timo Sepp\"al\"ainen}
\address{Timo Sepp\"al\"ainen\\ University of Wisconsin-Madison\\ 
Mathematics Department\\ Van Vleck Hall\\ 480 Lincoln Dr.\\  
Madison WI 53706-1388\\ USA.}
\email{seppalai@math.wisc.edu}
\urladdr{http://www.math.wisc.edu/~seppalai}
\thanks{T.~Sepp\"al\"ainen was partially supported by 
National Science Foundation grant DMS-0701091 and by the
Wisconsin Alumni Research Foundation.} 
\author[B.~Valk\'o]{Benedek Valk\'o}
\address{Benedek Valk\'o\\ University of Wisconsin-Madison\\ 
Mathematics Department\\ Van Vleck Hall\\ 480 Lincoln Dr.\\  
Madison WI 53706-1388\\ USA.}
\email{valko@math.wisc.edu}
\urladdr{http://www.math.wisc.edu/~valko}
\thanks{B.~Valk\'o  was partially supported by 
National Science Foundation grant DMS-0905820} 

\date{\today}
\begin{abstract} 
We study the scaling exponents of a  1+1-dimensional directed polymer in a Brownian  random
environment      introduced by O'Connell and  Yor.   For a version of the model with
  boundary conditions that are stationary in a space-time  sense we identify the
exact values of the exponents.  For the version without the boundary conditions we
get the conjectured upper bounds on the exponents.  
\end{abstract}
\maketitle
%

\section{Introduction}

We study the scaling exponents of a directed polymer model in 1+1  
  dimensions (one space dimension plus  time dimension)
   whose random environment is constructed from
Brownian motions.   For a positive integer $n$ and an inverse 
temperature  parameter $\beta>0$, 
the partition function is defined by 
\be
\fren_n(\beta)=   
\int\limits_{0<s_{1}<\dotsm<s_{n-1}<n}  \exp\bigl[ \beta\bigl(B_1(0,s_1) +B_{2}(s_1,s_{2})
+\dotsm + B_n(s_{n-1},n)\bigr)\bigr] \,ds_{1,n-1} 
\label{zetadef1}\ee
where $\{B_j\}$ are independent one-dimensional standard Brownian motions,
$B(s,t)=B(t)-B(s)$, and 
$ds_{1,n-1}$ $=$ $ds_1ds_2 \dotsm ds_{n-1}$.  
This model was   introduced
 by O'Connell and  Yor   \cite{oconn-yor-01} in connection with a related polymer model
 that they named  the  
{\sl generalized Brownian queue}.  Subsequently the exact limiting free energy density
${p(\beta)}$ was computed by
Moriarty and O'Connell \cite{mori-oconn-07}. To state their result, recall the  
  gamma function   $\Gamma(s)=\int_0^\infty x^{s-1}e^{-x}\,dx$ and the digamma 
function   $\digamf=\Gamma'/\Gamma$.

\begin{theorem} [\cite{mori-oconn-07}]    For $\beta>0$ this almost sure limit holds:  
\be \lim_{n\to\infty} n^{-1}\log \fren_n(\beta)=  {p(\beta)}=
\inf_{t>0}\left\{t \beta^2-\digamf(t)  \right\}-2\log\beta.  
\nn \ee
 \label{zeta1thm}
  \end{theorem}

Moriarty and O'Connell extracted this result with the help of large deviation asymptotics  from 
{the generalized Brownian queuing system} whose limit is readily computable.  
In this related model  the polymer path is allowed to begin in the
infinite past and a decaying exponential factor is included inside the integral
to make  the partition function converge.  We also work with the Brownian queueing system
  to obtain  an estimate on the fluctuations:
 
\begin{theorem} There exist finite, positive  $\beta$-dependent constants
$b_0, n_0, C$ such that for  $b\ge b_0$ and $n\ge n_0$ 
\[  \bP\bigl( \,\abs{\log \fren_n(\beta) -n {p(\beta)}}\ge bn^{1/3} \bigr)
\le Cb^{-3/2}.  \]
\label{zeta2thm}\end{theorem} 

The conjectured behavior for directed polymers  is that the 
  order of magnitude of the  fluctuations of $\log\fren_n(\beta)$ 
is  $n^\chi$ and in 1+1 dimensions this exponent takes the value 
 $\chi=1/3$  at all inverse temperatures $\beta>0$.  
Theorem \ref{zeta2thm} gives the expected  upper
bound on the exponent:  $\chi\le 1/3$.  In Theorem \ref{thm:freepath2} below
we give the corresponding upper bound on the fluctuations
of the polymer path.  

\medskip

 The generalized Brownian queueing system amounts
to putting boundary conditions on the polymer  \eqref{zetadef1}  that are stationary in a natural 
two-dimensional manner.   This stationarity comes from a 
  Burke-type  property  
discovered by O'Connell and  Yor   \cite{oconn-yor-01}, see Lemma \ref{lemburke} below. 
 For the  model with boundary conditions 
we identify the exact scaling exponents.  Our 
analysis of these models adapts  the steps of the recent work  \cite{sepp-poly} where a 
discrete lattice model with analogous properties was discovered and its 
scaling exponents studied.   The roots of the proofs in  \cite{sepp-poly}  can be traced back
to the seminal paper \cite{cato-groe-06}.  In the context of maximal increasing paths on planar
Poisson points, \cite{cato-groe-06}   was the first to chart a path to 
the scaling exponents  of a two-dimensional growth  model  without asymptotic
analysis of Fredholm determinants.

 Suboptimal but still highly nontrivial  
bounds on scaling exponents in 1+1 dimensional polymer models have been obtained
for Brownian polymers in Poisson environments 
 \cite{come-yosh-05,wuth-98aihp,wuth-98aop}, 
 for  Gaussian random walk    in a Gaussian environment 
   \cite{meja-04,petermann},  and  for the related model of  first passage percolation 
\cite{lice-newm-piza,newm-piza}.  

The overall situation in 1+1 dimensional polymers is now similar to that 
for two-dimensional directed last-passage percolation models,
which are of course closely related as zero-temperature directed  polymers.  In both
areas  there is a Brownian model and some particular discrete models that are amenable 
to explicit computations. Models with general distributions remain  beyond the reach of current techniques. 

Currently results are farther along for last-passage models:
 in addition to exponents, explicit Tracy-Widom  limit distributions are known.  
  Key results   appear in the papers 
 \cite{baik-deif-joha-99, bala-cato-sepp, bary-01, cato-groe-06,
ferr-spoh-06,grav-trac-wido-01, joha, joha-ptrf-00}.
The connection between last passage models and random matrix theory has been one
of the major inspirations of the subject.  The recent article \cite{oconn-toda} finds 
a connection between the Brownian polymer  \eqref{zetadef1}  and the quantum Toda lattice and 
proposes this as the possible polymer analogue of the random matrix connection
of last-passage percolation.   

There is one  directed polymer model which is essentially solved: the continuum random directed polymer in 1+1 dimension which describes a Brownian path in a white noise environment. 
 The free energy is defined as
\be \nonumber
\log E \left[:\!\exp\!:  \left\{-\int_0^T  \dot{\mathcal W}(t,b(t)) dt    \right\}     \right]
\ee
where $\dot{\mathcal W}$ is a space-time white noise,  $b$ is a Brownian motion (or bridge) and $:\!\!\exp\!\!:$ is the Wick-ordered exponential. In \cite{bala-quas-sepp} the exact scaling exponent is determined in the case of a stationary boundary condition using a connection to the KPZ equation and the
weakly asymmetric simple exclusion process. 
In \cite{amir-corw-quas} the same connection was used to compute the exact distribution of the free energy in the point-to-point setting, and the Tracy-Widom distribution is derived in the appropriate scaling limit. (See also \cite{sasa-spoh}).


\medskip

{\sl Some frequently used notation.} 
We write $f(s,t)=f(t)-f(s)$ for increments, without assuming that $s\le t$. 
$\bZ_+=\{0,1,2,\dotsc\}$, $\bN=\{1,2,3,\dotsc\}$ and $\bR_+=[0,\infty)$. 
  $\overline{X}=X-\E X$ denotes a centered random variable.  
In general a superscript $\om$ is added to a symbol whenever its dependence
on a particular realization of the environment needs to be made explicit.  

\section{Models and results}


We begin with precise definitions of two polymer models, the one already encountered
and another one with a boundary.  
$B$ and $\{B_k: k\in\bN\}$ 
 denote  independent  standard  Brownian motions indexed by $\bR$.
 They form the random environment $\om$ under probability measure $\bP$. 

By Brownian scaling $\{B(ct)\}\eqd\{c^{1/2}B(t)\}$  
the parameter $\beta$ can be removed  from the exponent in \eqref{zetadef1}
and replaced by a parameter that controls the upper limit of integration.  
This is   convenient for us,  so instead of  $\fren_n(\beta)$ 
 we work with the family 
\be Z_{j,k}(s,t) = 
\int\limits_{s<s_{j}<\dotsm<s_{k-1}<t}  \exp\bigl[ B_j(s,s_j) +B_{j+1}(s_j,s_{j+1})
+\dotsm + B_k(s_{k-1},t)\bigr] \,ds_{j,k-1}
\label{Zdef1.1}\ee
where $1\le j\le k\in \bN$ and $s<t\in\bR$.  Occasionally we may also write
$Z_{(j,k),(s,t)} =Z_{j,k}(s,t)$.  The distributional identity is 
\be \fren_n(\beta)\eqd\beta^{-2(n-1)}Z_{1,n}(0,{n}\beta^2).\label{brscal}\ee
 This is the last appearance of the partition function $Z_n(\beta)$ defined by \eqref{zetadef1}
 in the paper.   Similar notation will be used below for other partition functions.

The partition function is the normalizing constant for the quenched polymer 
distribution $Q_{(j,k),(s,t)}$.  This is a probability measure on 
nondecreasing cadlag paths  
  $x:[s,t]\to\{j,j+1,\dotsc,k\}$ that go from $x(s)=j$ to $x(t)=k$.  We represent these
paths   in terms of the jump times 
 $s<\sigma_j<\sigma_{j+1}<\dotsm<\sigma_{k-1}\le t $ {where}
$x(\sigma_i-) = i< i+1=x(\sigma_i)$.  
The   measure $Q_{(j,k),(s,t)}$ is defined   by 
\be\begin{aligned} 
&E^{Q_ {(j,k),(s,t)}}f(\sigma_j,\dotsc,\sigma_{k-1})  
=\frac1{Z_{{(j,k),(s,t)}}} 
\int\limits_{s<s_{j}<\dotsm<s_{k-1}<t}  f(s_j,\dotsc, s_{k-1}) 
 \\[7pt]
 &\qquad\qquad\qquad  \times  \exp\bigl[ B_j(s,s_j) +B_{j+1}(s_j,s_{j+1})
+\dotsm + B_k(s_{k-1},t) \bigr] \,ds_{j,k-1}.
\end{aligned}\label{Qdefjkst}\ee
This measure is called quenched because the environment of Brownian motions
is fixed.  Integrating away  the environment gives the 
annealed expectation  $E_{(j,k),(s,t)}(\cdot)=\bE E^{Q_{(j,k),(s,t)}}(\cdot)$.  

In addition to the digamma function we also need its derivative, the trigamma function $\trigamf=\digamf'$.
$\digamf$ is concave and increasing and $\trigamf$ is positive, convex
and strictly decreasing with $\trigamf(0+)=\infty$ and $\trigamf(\infty)=0$.
Theorem \ref{zeta1thm} is equivalent to the statement 
\be 
\lim_{n\to\infty}n^{-1}\log Z_{1,n}(0, n\tau)= \trigamf(\para)\para- \digamf(\para)
\quad \text{$\P$-a.s.}
\label{Zlim5}\ee
where   $\tau>0$ and   $\para$ is the unique value such that  
$\trigamf(\para)=\tau$. 
 Theorem \ref{zeta2thm} will be proved in Section \ref{sec:freeZ} in terms of 
$Z_{1,n}(0, n\tau)$.  In conjunction with Theorem \ref{zeta2thm} goes an upper
bound on the fluctuations of the path, also proved in Section \ref{sec:freeZ}.  

\begin{theorem}\label{thm:freepath2}
Let $\tau>0$ and    $0<\gamma<1$. Then for all large enough $n$
 and $b$    
\be 
P_{(1,n),(0,n\tau)}\bigl(|\sigma_{\lfloor n \gamma\rfloor}-n\gamma \tau |>b n^{2/3}   \bigr)\le 
C(\tau) b^{-3}.
\ee 
\end{theorem}

Theorem \ref{thm:freepath2} says that the path   stays close to the diagonal of the rectangle
$[1,n]\times[0,n\tau]$, and typical  fluctuations away from the diagonal have order of
magnitude at most $n^{2/3}$.  This gives an upper bound $\zeta\le 2/3$ for the second
basic scaling exponent $\zeta$  which describes the fluctuations of the polymer path. 

\medskip

These are the results for the polymer without a boundary, and we turn to discuss
the model with boundary.  For this model the upper bounds of Theorems 
\ref{zeta2thm} and \ref{thm:freepath2} are combined with matching lower bounds,
so we have the precise values of the scaling exponents.    An additional 
  parameter $\para>0$ is introduced in this model.  
 
The partition function is 
\be\begin{aligned} Z_{n}^\para(t) &= Z^\para_{n,t} =
\int_{-\infty<s_{0}<s_1<\dotsm<s_{n-1}<t}  \exp\bigl[ -B(s_0)+\para s_0 \\[7pt]
 &\qquad  +  B_1(s_0,s_1) 
+B_{2}(s_1,s_{2}) +\dotsm + B_n(s_{n-1},t)\bigr] \,ds_{0,n-1}.
\end{aligned}\label{Zdef2}\ee
 The quenched polymer measure $Q^\para_{n,t}$
lives on    nondecreasing cadlag paths    
  $x:(-\infty,t]\to\{0,1,\dotsc,n\}$ that go from $x(-\infty)=0$ to $x(t)=n$.  Again 
we represent these     in terms of  jump times 
 $-\infty<\sigma_0<\sigma_1<\dotsm<\sigma_{n-1}\le t $ {where}
$x(\sigma_i-) = i< i+1=x(\sigma_i)$.  
The   measure is defined   by 
\be\begin{aligned} 
&E^{Q^\para_{n,t}}f(\sigma_0,\sigma_1,\dotsc,\sigma_{n-1})  
=\frac1{Z_{n}^\para(t) } 
\int_{-\infty<s_{0}<\dotsm<s_{n-1}<t}  f(s_0, s_1,\dotsc, s_{n-1}) 
 \\[7pt]
 &\qquad  \times  \exp\bigl[ -B(s_0)+\para s_0 +  B_1(s_0,s_1) 
+B_{2}(s_1,s_{2}) +\dotsm + B_n(s_{n-1},t)\bigr] \,ds_{0,n-1}.
\end{aligned}\label{Qdef2}\ee
Annealed probability and expectation are   denoted by  
$P^\para_{n,t}(\cdot)=\bE Q^\para_{n,t}(\cdot)$  and 
 $E^\para_{n,t}(\cdot)=\bE E^{Q^\para_{n,t}}(\cdot)$, and simply
 $E(\cdot)$ when the parameters are understood. 
 
 \medskip
 
{\sl Notational remark.}   To simplify notation we shall not consistently 
carry the superscript $\para$ in the notation for the objects of the polymer model
with boundary.  The notational distinction between the two models is that 
the model without boundary has two space-time parameters represented by $((j,k),(s,t))$   in
definitions \eqref{Zdef1.1} and \eqref{Qdefjkst}, while the model with boundary
has only a single space-time parameter, namely $(n,t)$ in definitions
\eqref{Zdef2} and \eqref{Qdef2}.  When dependence on the environment $\om$ 
needs to be displayed explicitly, $\om$ is added as a superscript,  as   for example in 
$Q^{\para,\om}_{n,t}$.  
 
\medskip
 
The boundary conditions render the model stationary in a sense made precise
in Theorem \ref{burkethm}.  As a consequence some explicit computations can be
performed: from Theorems \ref{burkethm} and \ref{varthm} we obtain
\be
\E (\log Z_n^{\para}(t))=-n \digamf(\para)+\para t
\quad\text{and}\quad 
 \Vvv(\log Z_n^\para(t)) = n\trigamf(\para)-t+2 E^\para_{n,t}(\sigma_0^+).  
\label{meanvar}\ee

As we take the size of the polymer to
infinity,  the interesting exponents appear when the  endpoint follows approximately a   {\sl characteristic direction}
specified by the parameter $\para$.  The precise hypothesis is 
  \be \text{$t>0$ and $n\in\bN$ satisfy $\abs{t-n \trigamf(\para)}\le A n^{2/3} $  
  for a  constant $0\le A<\infty$. }   \label{assA}\ee 
The purpose of this assumption is to kill the first two terms of the variance formula in
\eqref{meanvar}.
The first theorem says that the exponent  $\chi=1/3$
describes  the order of fluctuations of $\log Z_n^{\para}(t)$.  

\begin{theorem}\label{varbdthm}  Assume \eqref{assA}.  
 Then there exist  constants $0<C_1 <C_2 <\infty$
 that depend on $(\para, A)$  such that for all
 $n\ge 1$ 
\be 
C_1 n^{2/3} \le \Vvv (\log Z_n^{\para}(t))\le  C_2 n^{2/3}.
\ee
\end{theorem}

The second theorem says that under the annealed distribution 
  the path stays close to the diagonal of the rectangle $[0,n]\times[0,t]$, and 
  $n^{2/3}$ is the correct order of typical fluctuations.   

\begin{theorem}\label{path-thm1} 
Assume \eqref{assA}   and   let $0\le \gamma<1$.   Then for
large enough $n$ and $b$ 
 \be \label{path11}
P^\para_{n,t}(\abs{\sigma_{\lfloor \gamma n \rfloor }-\gamma t}>b n^{2/3})\le C b^{-3}.
\ee 
 For any  $\varepsilon>0$ there exists $\delta>0$ such that  
\be\label{path21} 
\varlimsup_{n\to \infty}P^\para_{n,t}(|\sigma_{\lfloor \gamma n \rfloor }-\gamma t|\le \delta n^{2/3})\le \varepsilon.
\ee 
\end{theorem}

Variable  $\sigma_0$ is a special case of Theorem \ref{path-thm1}, but in fact
controlling this case turns out to be the key to {\sl both} Theorems  \ref{varbdthm} and
 \ref{path-thm1}.  Let us also mention that  our arguments give the 
 moment bound $E^\para_{n,t}(\abs{\sigma_0}^p)\le Cn^{2p/3}$ for $1\le p<3$.

\section{Properties of the model with boundary}

Following O'Connell and Yor \cite{oconn-yor-01} 
introduce the following processes to render the role of boundary conditions 
clearer in \eqref{Zdef2}: 
   $\uu_0(t)=B(t)$, and then define inductively for 
$k\in\bN$ 
\be\begin{aligned}
  r_k(t)&=\log\int_{-\infty}^t e^{\uu_{k-1}(s,t)-\para(t-s)+B_k(s,t)}\,ds
\end{aligned}\label{defrk}\ee
and 
\bea\begin{aligned}  
\uu_k(t)&= \uu_{k-1}(t)+r_k(0)-r_k(t),\\
X_k(t)&=B_{k}(t)+r_k(0)-r_k(t).
\end{aligned}\label{defYk}\eea
From the definitions {and a simple induction argument} follow the equations  
\be\begin{aligned} 
\sum_{k=1}^n r_k(t)
&= \log \int_{-\infty<s_{0}<s_1\dotsm<s_{n-1}<t}  \exp\bigl[ B(s_0,t)-\para(t- s_0) \\[7pt]
 &\qquad \qquad +  B_1(s_0,s_1) 
+B_{2}(s_1,s_{2}) +\dotsm + B_n(s_{n-1},t)\bigr] \,ds_{0,n-1}\\[4pt]
&= B(t)-\para t +\log Z_{n}^\para(t) 
\end{aligned} \label{rB1}\ee
and 
\be\begin{aligned} 
Z_{n}^\para(t) =  \int_0^t   \exp[- B(s)+\para s] Z_{1,n}(s,t) \,ds
\;+ \; \sum_{j=1}^n  \Bigl( \;\prod_{k=1}^j e^{r_k(0)} \Bigr) Z_{j,n}(0,t). 
\end{aligned}\label{ZB2}\ee
If we define  also 
\be  Z_{0}^\para(t)= \exp[ -B(t)+\para t]  \label{Z0}\ee
then the above can  be expressed in the form
\be\begin{aligned} 
Z_{n}^\para(t) =  \int_0^t  Z_{0}^\para(s) Z_{1,n}(s,t) \, ds
\; + \; \sum_{j=1}^n  Z_{j}^\para(0) Z_{j,n}(0,t). 
\end{aligned}\label{ZB3}\ee
Combining \eqref{rB1} with \eqref{defYk}  gives the space and time increments 
\be 
\log Z_{n}^\para(t)- \log Z_{n-1}^\para(t) = r_n(t)  \qquad
 \text{for $n\in\bN,\, t\in\bR_+$}  \label{Zincr1}\ee
and
\be 
\log Z_{n}^\para(t)- \log Z_{n}^\para(s) = \para(t-s)-\uu_n(s,t)  \qquad
 \text{for $n\in\bZ_+,\, s<t\in\bR_+$} .  \label{Zincr2}\ee

\begin{remark}\label{remark1}
The density of $(\sigma_0,\dots, \sigma_{n-1})$ under $Q_{n,t}$ can also be written as
\be
\frac1{\widehat Z_n^\para(t)}
\exp\bigl[
\widehat B(s_0,t)+\widehat B_1(s_0,s_1)+\widehat B_2(s_1,s_2)+\dots+\widehat B_n(s_{n-1},t)
\bigr]\ind\{s_0<\dotsm<s_{n-1}<t\}
\label{altQ}\ee
where $\widehat B(u)=B(u)-\para u/2$ (and similarly for $\widehat B_k$) 
and $\widehat Z_n^\para(t)=Z_n^\para(t) \exp(B(t)-\para t)$. From this representation
and the stationarity of Brownian increments it is clear that 
\be
E^{Q_{n,t}}f(\sigma_0,\sigma_1,\dotsc,\sigma_{n-1}) \eqd 
E^{Q_{n,0}}f(t+\sigma_0,t+\sigma_1,\dotsc,t+\sigma_{n-1}).
\label{altQ2} \ee
\end{remark}

\subsection{Burke property}


The following lemma summarizes the Burke property of O'Connell and Yor  \cite[Thm.~5]{oconn-yor-01}.
\begin{lemma}[\cite{oconn-yor-01}]\label{lemburke} Let $B$ and $C$ be independent standard Brownian motions indexed by $\bR$ and $\para>0$ a fixed constant. For $t\in\bR$ set 
\begin{align*}
&r(t)=\log\int_{-\infty}^t \exp\left(B(s,t)+C(s,t)-\para(t-s)\right)ds,&\\
&f(t)=B(t)+r(0)-r(t), \quad g(t)=C(t)+r(0)-r(t).&
\end{align*}
Then $f$ and $g$ are independent standard Brownian motions indexed by $\bR$ and  for each $t\ge0$  the processes $\{(f(s),g(s)):s\le t\}$ and $\{r(s):s\ge t\}$ are independent. Moreover the following identity holds almost surely:
\be\label{r_ident}
r(t)=\log\int^{\infty}_t \exp\left(f(t,s)+g(t,s)+\para(t-s)\right)ds.
\ee
\end{lemma}
Although   identity (\ref{r_ident}) is not stated explicitly in Theorem 5 of \cite{oconn-yor-01} it is 
contained in the proof.
In the next theorem we generalize the Burke property in a form suitable for our use.   

\begin{theorem}\label{burkethm}
Let $n\in\bN$ and  $0\le s_n\le s_{n-1}\le\dots\le s_1<\infty$. 
Then over the index $j$ 
 the following random variables and processes are all mutually independent:
 \begin{align*}
& r_j(s_j) \ \text{ and } \    \{X_j(s): s\le s_j\} \ \text{ for } \    1\le j\le n, \quad  \{Y_n(s): s\le s_n\}, \quad \\
& \qquad\qquad\qquad\qquad\qquad   \text{and} \quad
\{Y_j(s_{j+1},s): s_{j+1}\le s\le s_j\} \ \text{ for } \    1\le j\le n-1.
 \end{align*}
Furthermore,  the $X_j$ and $Y_j$  processes are standard Brownian motions,  and $r_j(s_j)\overset{d}=-\log \eta$ with $\eta\sim Gamma(\para,1)$. 
\end{theorem}
\begin{proof}
We use induction on $n$.   For $n=1$ the statement follows from Lemma \ref{lemburke} together with Dufresne's identity \cite[Cor.~4]{dufr-osaka01} which gives the distribution of $r_1(t)$.


We now assume that the statement is true for $n-1$ and prove it for $n$. From definitions \eqref{defrk}--\eqref{defYk},  
$\{r_{n}(s), Y_{n}(s), X_n(s): s\le s_n\}$ is a function of 
 $\{Y_{n-1}(s): s\le s_{n}\}$ and an
independent Brownian motion $B_{n}$. This means that $\{r_{n}(s), Y_{n}(s), X_n(s): s\le s_n\}$ is independent of  the processes $\{X_j(s): s\le s_j\}$, $\{Y_j(s_{j+1},s): s_{j+1}\le  s\le s_j\}$
for  $1\le j\le n-1$, and random variables $\{r_j(s_j): 1\le j \le n-1\}$.
 Besides the induction hypothesis we also used that $\{Y_{n-1}(s_{n},s), s\ge s_{n}\}$ is independent of  $\{Y_{n-1}(s): s\le s_{n}\}$. That $\{ Y_{n}(s), X_n(s): s\le s_n\}$ and $ r_{n}(s_n)$ have the right joint distribution again follows from Lemma \ref{lemburke} and Dufresne's identity.
\end{proof}
Applying the shift $\shift_uf(t)=f(u+t)-f(u)$ that preserves
the distributions of the  Brownian motions 
shows that  $r_k(t)=r_k(0)\circ\shift_t$ is a stationary process. {We also note that $\bE (r_k(t))=-\digamf(\para)$ and $\Vvv (r_k(t))=\trigamf(\para)$.}
\subsection{Reversal}

Fix $T\in\bR$ and  $n\in \bN$.
Define the following new processes:
\bea 
\label{defY*} Y_j^*(s)&=&Y_{n-j}(T)-Y_{n-j}(T-s), \qquad 0\le j\le n,\\
\label{defB*}B_j^*(s)&=&X_{n+1-j}(T)-X_{n+1-j}(T-s),\qquad 1\le j\le n,\\
\label{defr*}r_j^*(s)&=&r_{n+1-j}(T-s),\qquad 1\le j\le n,\\
\label{defX*}X_j^*(s)&=&B_{n+1-j}(T)-B_{n+1-j}(T-s),\qquad 1\le j\le n.
\eea 
Define the {\sl dual   environment}  by  $\om^*=(Y^*_0, B^*_j: 1\le j\le n)$.  
\begin{theorem}\label{revthm}  Fix $n\in\bN$.  
We have the following equality in distribution for the  processes on $\bR$:
\be \begin{aligned} 
&\left\{Y_{k}: 0\le k\le n; \, X_j, B_{j}, r_j: \, 1\le j\le n \right\}  \eqd 
\left\{Y_{k}^*: 0\le k\le n; \, X_j^*, B_j^*, r_j^*: \, 1\le j\le n \right\}.
\end{aligned} \notag\ee 
In particular, the dual environment $\om^*=(Y^*_0, B^*_j: 1\le j\le n)$ 
has the same distribution as the original environment $\om=(B, B_j: 1\le j\le n)$.
\end{theorem}
\begin{proof}
{By (\ref{r_ident}) from Lemma \ref{lemburke}}
\be \label{rXY} 
r_k(t)=\log \int_{t}^{\infty} \exp(Y_k(t,s)+\para(t-s)+X_k(t,s)) ds.
\ee 
Definitions (\ref{defY*})--(\ref{defX*}) together with (\ref{defYk}) and 
\eqref{rXY}  give the identities
\be \begin{aligned} 
r_k^*(t)&=\log\int_{-\infty}^t \exp(Y^*_{k-1}(s,t)-\para(t-s)+B_k^*(s,t))\,ds,\\
Y^*_k(t)&= Y^*_{k-1}(t)+r^*_k(0)-r^*_k(t),\\[3pt]
X^*_k(t)&=B_k^*(t)+r^*_k(0)-r^*_k(t). 
\end{aligned}  \label{eqn*}\ee 
  Theorem \ref{burkethm} 
  tells us that  $\{Y_n, X_j, 1\le j\le n\}$  are independent Brownian motions on $(-\infty, t]$ 
for any $t>0$, hence over all of $\bR$.   Consequently by \eqref{defY*}--\eqref{defB*} processes 
   $\{Y_0^*, B_j^*, 1\le j\le n\}$  have the same distribution as $\{Y_0, B_j, 1\le j\le n\}$.
The theorem follows because  \eqref{eqn*} defines the same recursions as  \eqref{defrk}--\eqref{defYk}.
\end{proof}

We define the dual quenched measure $Q^*_n$   on non-decreasing cadlag paths $x:\bR_+\to \{0,1,\dots,n\}$ with $x(0)=0$ and $x(\infty)=n$. These paths can be represented by    jump times 
$0<\sigma^*_1<\dotsm<\sigma^*_{n}$ defined by $x(\sigma^*_j-)=j-1<j=x(\sigma^*_j)$.  The dual measure   $Q^*_n$  is defined by 
 \be\label{Qstar}\begin{aligned} 
&E^{Q^*_{n}}\bigl[f(\sigma^*_1,\sigma^*_2,\dotsc,\sigma^*_{n})\bigr]  =\frac1{Z_n^*} 
\int_{0<s_1<\dotsm<s_n<\infty}  f(s_1, s_2,\dotsc, s_n)   \\[4pt]
&\qquad \times\exp\left[ X_1(0,s_1)+X_2(s_1,s_2)+\dotsm+ X_n(s_{n-1},s_{n})+Y_n(0,s_{n})-\para s_{n}\right]
\,d{s_{1,n}}.  
\end{aligned}\ee
Because of the shifting described in \eqref{altQ2} we only need 
the measure $Q^*_{n}$ instead of a family indexed also by $t$. 
 The dual measure is naturally connected  with the representation of  $Q_{n,t}$ given
 in Remark \ref{remark1}.   The    next lemma is proved by a straightforward  change of
 variables in the integrals.  {Recall the definition of $\widehat Z_n^\para(t)$ from (\ref{altQ}).}
\begin{lemma}\label{LEMrev}  Fix $t\in\bR$ and $n\in\bN$  and for each
$\om$  define the dual environment $\om^*$ 
with $T=t$.  
Then  $
Z_n^{*,\om} =  \widehat Z_{n}^{\para, \om^*}(t), 
$
and 
\[  E^{Q^{*,\om}_{n}}\bigl[f(\sigma^*_1,\sigma^*_2,\dotsc,\sigma^*_{n})  \bigr]
=    E^{Q_{n,t}^{\om^*}}\bigl[ f(t-\sigma_{n-1},t-\sigma_{n-2},\dotsc,t-\sigma_{0})\bigr].    \]  
Consequently, by Theorem \ref{revthm}, 
\[
 E^{Q^{*}_{n}}\bigl[f(\sigma^*_1,\sigma^*_2,\dotsc,\sigma^*_{n}) \bigr]
   \eqd   E^{Q_{n,t}}\bigl[f(t-\sigma_{n-1},t-\sigma_{n-2},\dotsc,t-\sigma_{0})\bigr].  
\]
 \end{lemma}

\subsection{Variance identity}

\begin{theorem} For $\para>0$ and  $(t,n)\in(0,\infty)\times\bN$
\be\begin{aligned}
 \Vvv[\log Z_n^\para(t)] &= n\trigamf(\para)-t+2 E^\para_{n,t}(\sigma_0^+)\\
&= -n\trigamf(\para)+t+2 E^\para_{n,t}(\sigma_0^-). 
\end{aligned}\label{var1}\ee
\label{varthm}\end{theorem}
\begin{remark}
Adding the two equations gives $ \Vvv[\log Z_n^\para(t)] =E^\para_{n,t}|\sigma_0|$ while subtracting them yields $E^\para_{n,t} (\sigma_0)=t-n\trigamf(\para)$.
\end{remark}
\begin{proof}[Proof of Theorem \ref{varthm}]  While we keep $\para$ fixed we simplify the notation
to $Z_n(t)=Z_n^\para(t)$. 

The proof begins the same way as the proof of 
Theorem 3.7 in \cite{sepp-poly}. (The idea was
originally learned from the proof of Theorem 2.1 in \cite{cato-groe-06}.)
Abbreviate the increments with reference to compass directions:
\begin{align*}
\north &=\log Z_{n}(t)-\log Z_{n}(0), 
\quad \south =\log Z_{0}(t)=\para t-B(t),\\
\quad \east &=\log Z_{n}(t)-\log Z_{0}(t),  
\quad \west =\log Z_{n}(0)=\sum_{j=1}^n r_j(0). 
\end{align*}
The south and west increments are given by the boundary conditions
as indicated above, while the east and {north} increments are computed 
from \eqref{ZB2}. 
Then 
\be\begin{aligned}
\Vvv\bigl[\log Z_{n}(t)\bigr]&=\Vvv(\west +\north )=\Vvv(\west )+\Vvv(\north )+2\Cvv(\west ,\north )\\
&=\Vvv(\west )+\Vvv(\north )+2\Cvv(\south +\east -\north ,\north )\\
 &= \Vvv(\west )-\Vvv(\north )+2\Cvv(\south ,\north )\\
&=n\trigamf(\para)-t+2\Cvv(\south ,\north ).\\
\end{aligned}\label{aux3}\ee
{The third line comes from the independence of $\east $ and $\north $ while the  last equality came 
 from (\ref{Zincr2}) and Theorem \ref{burkethm}.} Substituting $\north$ instead of
$\west$ in the covariance in the second line of (\ref{aux3}) leads to the formula 
\be
\Vvv\bigl[\log Z_{n}(t)\bigr] =  -n\trigamf(\para)+t
+2\Cvv(\east,\west ).
\label{aux3.1}\ee
Equations \eqref{aux3} and \eqref{aux3.1} give the two lines of 
\eqref{var1} once we evaluate the covariances.  
We begin with \eqref{aux3}. 

We need to vary separately the parameters of the boundary conditions
on the $x$ and $y$ axes.  So we rename the parameter on the
$y$-axis as $\paraa$, and rewrite the partition function with
boundaries 
as follows:
\be\begin{aligned} Z_{n}(t) &= 
\int_{0<s_{0}<s_1\dotsm<s_{n-1}<t}  \exp\bigl[ -B(s_0)+\para s_0 \\[7pt]
 &\qquad\qquad   +  B_1(s_0,s_1) 
+B_{2}(s_1,s_{2}) +\dotsm + B_n(s_{n-1},t)\bigr] \,ds_{0,n-1}\\
&
+ 
\int_{-\infty<s_{0}<s_1\dotsm<s_{n-1}<t} \ind\{s_0<0\} 
 \exp\bigl[ -B(s_0)+\paraa s_0 \\[7pt]
 &\qquad \qquad +  B_1(s_0,s_1) 
+B_{2}(s_1,s_{2}) +\dotsm + B_n(s_{n-1},t)\bigr] \,ds_{0,n-1}.
\end{aligned}\label{Z5}\ee

Next we argue that 
\be \text{$\bE(\north\vert\south=x)$ does not
depend on $\para$.}  \label{aux3.8}\ee
%
Indeed, if we condition $\{h(s)=\para s-B(s), 0\le s\le t\}$ on the event $h(t)=x$ then this (as a process) has the same distribution as $\para s-\tilde B_s$ where $\tilde B_s$ is a BM conditioned to hit $y=\para t-x$ at $t$. It is known that the latter has the same distribution as $B_s-sB_t/t  +{sy}/{t}$ which means that the conditional distribution of $\{h(s), 0\le s\le t\}$ is the same as
\[
\para s-\left(B_s- \frac{sB_t}{t}+\frac{s(\para t-x)}{t}\right)=\frac{s}{t} B_t-B_s+\frac{sx}{t}
\]
which does not depend on $\para$. 

The density of $\south$ is 
$f_\para(x)=(2\pi t)^{-1/2} \exp(-\frac1{2t}(x-\para t)^2)$. 
Utilizing \eqref{aux3.8},  
\be\begin{aligned}
\frac{\partial}{\partial\para} \E(\north )
&=\frac{\partial}{\partial\para} \int_\bR \E(\north \,\vert\, \south =x) f_\para(x)\,dx
=\int_\bR \E(\north \,\vert\, \south =x)  
\frac{\partial   f_\para(x)}{\partial\para}\,dx\\
&=\int_\bR \E(\north \,\vert\, \south =x) ( x-\para t)f_\para(x)\,dx\\
&=\E(\north \south )- \E(\north )\E(\south) =
\Cvv(\north ,\south ). 
\end{aligned}\label{aux4}\ee
On the other hand, utilizing \eqref{Z5},
\be\begin{aligned}
&\frac{\partial}{\partial\para} \E(\north )
= \E\Bigl[\; \frac{\partial}{\partial\para} \log Z_n(t)\Bigr] \\
&=\E\Bigl[\; \frac1{ Z_n(t)} \int_{0<s_{0}<s_1\dotsm<s_{n-1}<t} 
s_0\, \exp\bigl\{ -B(s_0)+\para s_0 \\[7pt]
 &\qquad\qquad   +  B_1(s_0,s_1) 
+B_{2}(s_1,s_{2}) +\dotsm + B_n(s_{n-1},t)\bigr\} \,ds_{0,n-1}\Bigr]\\
&=E(\sigma_0^+). 
\end{aligned}\label{aux5}\ee
Combining \eqref{aux4}, \eqref{aux4} and \eqref{aux5} gives
the first line of  \eqref{var1}.

Proof of the second line of \eqref{var1} proceeds analogously. 
$\bE(\east\,\vert\,\west=x)$ does not
depend on $\paraa$, and so the analogue of computation 
\eqref{aux4} gives $\partial_\paraa \E(\east) = -\Cvv(\east, \west)$. 
Then from  \eqref{Z5},  $\partial_\paraa \E(\east) 
=\E(\partial_\paraa \log Z_n(t))
=E(\sigma_0\ind\{\sigma_0<0\})$. 
\end{proof}
 
\subsection{Comparison lemma}

We find it useful here  to augment the family $Z_{j,k}(s,t)$ defined
for $j\ge 1$  by \eqref{Zdef1.1}
by introducing   $Z_{0,k}(t)=Z_{0,k}(0,t)$. These will be defined not exactly consistently 
with \eqref{Zdef1.1}, but in a manner that gives us inequalities between ratios of 
partition functions.  
For $k\in \bN$ and $t\in\bR_+$ define 
\be  Z_{0,0}(t)=e^{-B(t)} \label{Zdef1.01} \ee
and 
 \be\begin{aligned} Z_{0,k}(t) &= 
\int\limits_{0<s_{0}<\dotsm<s_{k-1}<t}  \exp\bigl[ -B(s_0) +B_1(s_0,s_1) 
\\[4pt]   &\qquad \qquad \qquad \qquad  
+B_{2}(s_1,s_{2})
+\dotsm + B_k(s_{k-1},t)\bigr] \,ds_{0,k-1}.
\end{aligned}\label{Zdef1.02}\ee
 
For $n\in\bN$ and events $D$ on the paths 
 we write $Z_{n,t}^\para(D)=Z_{n,t}^\para Q_{n,t}^\para(D)$ for the unnormalized
quenched measure.  It is also convenient  to set, for $A\subseteq\bR$,  
\be Z_{0,t}^\para(\sigma_0\in A)= \ind_{A\cap\bR_+}(t) \exp[-B(t)+\para t].  \label{0conv}\ee 

\begin{lemma}  Let $\para>0$.  For $0< s<t$ and $n\in\bZ_+$
\be
\frac{Z_{n+1,t}^\para(\sigma_0>0)}{Z_{n,t}^\para(\sigma_0>0)}
\le \frac{Z_{0,n+1}(t)}{Z_{0, n}(t)}\le \frac{Z_{n+1,t}^\para(\sigma_0<0)}{Z_{n,t}^\para(\sigma_0<0)}
\label{comp1}\ee
and 
\be
\frac{Z_{n,t}^\para(\sigma_0>0)}{Z_{n,s}^\para(\sigma_0>0)}
\ge \frac{Z_{0,n}(t)}{Z_{0, n}(s)}\ge \frac{Z_{n,t}^\para(\sigma_0<0)}{Z_{n,s}^\para(\sigma_0<0)}. 
\label{comp2}\ee
The second inequality of \eqref{comp2} makes sense only for $n\ge 1$. 
\end{lemma} 
\begin{proof} 
We check the  cases that initialize the inductive proofs.   For \eqref{comp1} 
\begin{align*}
\frac{Z_{1,t}^\para(\sigma_0>0)}{Z_{0,t}^\para(\sigma_0>0)} 
&=\frac{\int_0^t e^{-B(s)+\para s+B_1(s,t)}\,ds}{e^{-B(t)+\para t}}
=\int_0^t e^{B(s,t)+\para (s-t)+B_1(s,t)}\,ds\\
&\le  \int_0^t e^{B(s,t)+B_1(s,t)}\,ds  = \frac{Z_{0,1}(t)}{Z_{0, 0}(t)}  
<\infty = \frac{Z_{1,t}^\para(\sigma_0<0)}{Z_{0,t}^\para(\sigma_0<0)} . 
\end{align*}
For the first part of  \eqref{comp2} 
\begin{align*}
\frac{Z_{0,t}^\para(\sigma_0>0)}{Z_{0,s}^\para(\sigma_0>0)}=e^{-B(s,t)+\para (t-s)}
>e^{-B(s,t)}=  \frac{Z_{0,0}(t)}{Z_{0, 0}(s)}
\end{align*} 
and for the second part (now for $n=1$) 
\begin{align*}
 \frac{Z_{0,1}(t)}{Z_{0, 1}(s)} = e^{B_1(s,t)} 
 +\frac{\int_s^t e^{-B(u)+ B_1(u,t)}\,du}{\int_0^s e^{-B(u)+ B_1(u,s)}\,du} 
 >e^{B_1(s,t)}  =  \frac{Z_{1,t}^\para(\sigma_0<0)}{Z_{1,s}^\para(\sigma_0<0)}.
\end{align*} 

Next the induction steps.  We make use of the decomposition (for $0\le s<t$) 
\[
 {Z_{n,t}^\para(\sigma_0\in A)} 
=
 {Z_{n,s}^\para(\sigma_0\in A)} e^{B_n(s,t)} 
+ \int_s^t   {Z_{n-1,u}^\para(\sigma_0\in A)}  e^{B_n(u,t)} \,du
\]
valid for $n\in\bN$ and for both $A=(-\infty, 0)$ and $A=(0,\infty)$.  For $n=1$ 
convention \eqref{0conv} is used on the right-hand side.  

We begin with  the first inequalities in  \eqref{comp1}--\eqref{comp2}. 
 Assume the first inequality of both  \eqref{comp1} and \eqref{comp2} holds for $n-1$. 
We verify first \eqref{comp2} for $n$. 
\begin{align*}
\frac{Z_{n,t}^\para(\sigma_0>0)}{Z_{n,s}^\para(\sigma_0>0)}
&=  e^{B_n(s,t)} 
+ \int_s^t  \frac{Z_{n-1,u}^\para(\sigma_0>0)}{Z_{n,s}^\para(\sigma_0>0)} \,e^{B_n(u,t)} \,du\\
&=  e^{B_n(s,t)} 
+ \int_s^t  \frac{Z_{n-1,u}^\para(\sigma_0>0)}{Z_{n-1,s}^\para(\sigma_0>0)}
\cdot  \frac{Z_{n-1,s}^\para(\sigma_0>0)}{Z_{n,s}^\para(\sigma_0>0)} \,e^{B_n(u,t)} \,du\\
&\ge   e^{B_n(s,t)} 
+ \int_s^t  \frac{Z_{0,n-1}(u)}{Z_{0,n-1}(s)} 
\cdot  \frac{Z_{0,n-1}(s)}{Z_{0,n}(s)} \,e^{B_n(u,t)} \,du\\
&= \frac{Z_{0,n}(t)}{Z_{0, n}(s)}
\end{align*}
and the inequality came from the induction assumption. 

Deriving the first inequality of  \eqref{comp1} for $n$ requires an extra step.  First decompose and use 
 \eqref{comp2} for $n$ that we just proved:  
\begin{align*}
&\frac{Z_{n+1,t}^\para(\sigma_0>0)}{Z_{n,t}^\para(\sigma_0>0)} 
= \frac{Z_{n+1,s}^\para(\sigma_0>0)}{Z_{n,t}^\para(\sigma_0>0)}\,e^{B_n(s,t)} 
+ \int_s^t  \frac{Z_{n,u}^\para(\sigma_0>0)}{Z_{n,t}^\para(\sigma_0>0)} \,e^{B_n(u,t)} \,du\\[7pt]
&\qquad\qquad= \frac{Z_{n+1,s}^\para(\sigma_0>0)}{Z_{n,s}^\para(\sigma_0>0)} 
\cdot \frac{Z_{n,s}^\para(\sigma_0>0)}{Z_{n,t}^\para(\sigma_0>0)}\,e^{B_n(s,t)} \,
+ \int_s^t  \frac{Z_{n,u}^\para(\sigma_0>0)}{Z_{n,t}^\para(\sigma_0>0)} \,e^{B_n(u,t)} \,du\\
&\qquad\qquad\le  \frac{Z_{n+1,s}^\para(\sigma_0>0)}{Z_{n,s}^\para(\sigma_0>0)} 
\cdot \frac{Z_{0,n}(s)}{Z_{0,n}(t)}\,e^{B_n(s,t)} \,
+ \int_s^t  \frac{Z_{0,n}(u)}{Z_{0,n}(t)}  \,e^{B_n(u,t)} \,du\\
&\qquad\qquad\equiv \frac{Z_{n+1,s}^\para(\sigma_0>0)}{Z_{n,s}^\para(\sigma_0>0)} \,
a(s,t) + \int_s^t a(u,t)\,du. 
\end{align*}
The last equality defines the continuous function $a(s,t)$.  The same steps applied
to the middle member of  \eqref{comp1} gives an equality
\be 
\frac{Z_{0,n+1}(t)}{Z_{0, n}(t)}  = \frac{Z_{0,n+1}(s)}{Z_{0, n}(s)}
\,a(s,t) + \int_s^t a(u,t)\,du \notag\ee
from which 
\be
\frac{Z_{n+1,t}^\para(\sigma_0>0)}{Z_{n,t}^\para(\sigma_0>0)}  -  \frac{Z_{0,n+1}(t)}{Z_{0, n}(t)} 
\; \le \; a(s,t) \biggl( \frac{Z_{n+1,s}^\para(\sigma_0>0)}{Z_{n,s}^\para(\sigma_0>0)} 
- \frac{Z_{0,n+1}(s)}{Z_{0, n}(s)}\biggr).   \label{compaux8}\ee
Each of the three terms on the right-hand side above vanishes as we take $s\searrow 0$.  We have proved that the  first inequalities in  \eqref{comp1}--\eqref{comp2}  hold for $n$. 

To derive the second inequalities in  \eqref{comp1}--\eqref{comp2}   for $n$ 
the induction proofs work the same way: once $(\sigma_0>0)$ has been replaced
by $(\sigma_0<0)$   all inequalities in the above calculations have to be reversed. 
In \eqref{compaux8}  on the right the ratio  $Z_{n+1,s}^\para(\sigma_0<0)/Z_{n,s}^\para(\sigma_0<0)$
does not vanish as $s\searrow 0$ but it stays bounded a.s.~and again $a(0+,t)=0$ takes
the entire right-hand side to $0$.  
\end{proof}
 
\section{Upper bound for the variance }
The parameter $\para\in(0,\infty)$ can be considered  fixed throughout.  The constants
such as 
$C(\para)$ that depend on $\para$ 
and  appear in all our statements are locally bounded functions of $\para$.  {Sometimes we will suppress the dependence on $\para$ and the constants in the intermediate estimates may change from line to line.}
Recall also the key assumption \eqref{assA} on $t>0$ and $n\in\bN$, namely that 
$\abs{t-n \trigamf(\para)}\le A n^{2/3} $.  
We develop the upper bounds so that the effect of the constant  $A$ is explicitly present.  

\begin{theorem}\label{UBthm}  Assume \eqref{assA}.  
 Then there exists a constant $C(\para)<\infty$ such that 
\be 
\Vvv (\log Z_n^{\para}(t))\le (9A+C(\para)) n^{2/3}.
\ee
\end{theorem}
The value 9 in the statement has no particular meaning, except as a constant
independent of $\para$.  
 As a first step we give a bound on the tails of $\sigma_0^\pm$.
 
\begin{lemma}\label{lem_sigtail} Assume \eqref{assA} and fix $K>0$. Then
we can fix 
$C=C(\para)<\infty$ large enough  and $s=s(K,\para)>0$ small enough   such that,
 for $3An^{2/3}\le u\le Kn $,  
we have the bound  
\be \P(Q_{n,t}(\sigma_0^\pm\ge u)\ge  e^{-s u^2/n})\le \frac{C n^2}{u^4} 
E_{n,t} (\sigma_0^\pm) + \frac{C n^2}{u^3}. \label{sigtail_a}
\ee
 \end{lemma}
\begin{proof}We first consider the proof for $\sigma^+$.  
Set $\lambda=\para+\frac{bu}{n}$ where $b>0$ will be small.
Superscripts $\para$ and $\lambda$ indicate which parameter is used.  
 We have
\bea
Q^\para_{n,t}(\sigma_0^+\ge u)&=&\frac1{Z_{n}^\para(t) } \nn
\int_{-\infty<s_{0}<\dotsm<s_{n-1}<t}  \m1(s_0\ge u)
 \\
 &&\qquad  \times  \exp\bigl[ -B(s_0)+\para s_0 +  B_1(s_0,s_1) 
+\dotsm + B_n(s_{n-1},t)\bigr] \,ds_{0,n-1}.\nn\\
&\le&\frac1{Z_{n}^\para(t) } \nn
\int_{-\infty<s_{0}<\dotsm<s_{n-1}<t}  \m1(s_0\ge u) e^{(\para-\lambda)u}
 \\
 &&\qquad  \times  \exp\bigl[ -B(s_0)+\lambda s_0 +  B_1(s_0,s_1) 
+\dotsm + B_n(s_{n-1},t)\bigr] \,ds_{0,n-1}.\nn\\
&\le& \frac{Z_n^{\lambda}(t)}{Z_{n}^\para(t) }e^{(\para-\lambda) u}.\nn
\eea
Thus
\bea \nn
\P(Q^\para_{n,t}(\sigma_0^+\ge u)\ge e^{-su^2/n})&\le& 
\P\bigl(\,\frac{Z_n^{\lambda}(t)}{Z_{n}^\para(t) }e^{(\para-\lambda) u}\ge e^{-su^2/n}\bigr)\\
&=&\P(\log( Z_n^{\lambda}(t))-\log( Z_n^{\para}(t))\ge (\lambda-\para)u-su^2/n)\nn
\eea
From \eqref{rB1} and $e^{-r_k(t)}\sim$ Gamma($\para$,1) (Theorem \ref{burkethm}) 
\be\label{ElogZ}
\E \log( Z_n^{\para}(t))=-n \digamf(\para)+\para t
\ee
so upon centering 
\bea\nn 
&&\P(Q^\para_{n,t}(\sigma_0^+\ge u)\ge e^{-su^2/n})\le \\
\nn
&&\hskip15mm\P\bigl\{\overline{\log( Z_n^{\lambda}(t))}-\overline{\log( Z_n^{\para}(t))}\ge
n(\digamf(\lambda)-\digamf(\para))-{t}(\lambda-\para)+(\lambda-\para)u-su^2/n).  
\eea
Given $K$ we can restrict $b=b(K,\para)$ small enough so that 
$\abs{\lambda-\para}\le bK\le \para/2$ and then a 
  Taylor expansion gives 
\[
|\digamf(\lambda)-\digamf(\para)-(\lambda-\para) \trigamf(\para)|\le C(\para) (\lambda-\para)^2.  
\]
Together with $|t-n\trigamf(\para)|\le A n^{2/3}$ this leads to
\begin{align*}
&n(\digamf(\lambda)-\digamf(\para)-\frac{t}{n}(\lambda-\para))+(\lambda-\para)u-su^2/n\\
&\qquad \ge
-C(\para) b^2 u^2/n-A b u n^{-1/3}+  b u^2/n-s u^2/n\\
&\qquad\ge \frac{bu^2}{3n} \bigl(1-3C(\para)b\bigr)  +  \frac{bu}{3n^{1/3}}(un^{-2/3}-3A)   
 + \frac{bu^2}{3n}( 1-3s/b)  \\
&\qquad \ge C(\para) u^2/n.  
\end{align*}
The last line, with  $C(\para)>0$,  follows  by enforcing $u\ge 3An^{2/3}$, 
  choosing $s=b/3$ and then taking $b=b(\para)$ small enough. Put  this back above and apply Chebyshev's inequality:
\bea \nn
\P(Q^\para_{n,t}(\sigma_0^+\ge u)\ge e^{-su^2/n})&\le&\P(\overline{\log( Z_n^{\lambda}(t))}-\overline{\log( Z_n^{\para}(t))}\ge C u^2/n)\\ \nn
 &\le& C \frac{n^2}{u^4}
\Vvv\bigl[\log Z_n^{\lambda}(t)-\log Z_n^{\para}(t)\bigr]\\ \label{anotherUP}
&\le& 2C \frac{n^2}{u^4}\bigl(\Vvv[\log( Z_n^{\lambda}(t))]+\Vvv[\log( Z_n^{\para}(t))] \bigr).
\eea
{Using Lemma \ref{4.1lem} below together with  the definition of  $\lambda$:}
\begin{align*}
\Vvv(\log( Z_n^{\lambda}(t))&\le \Vvv(\log( Z_n^{\para}(t)) + n\abs{\trigamf(\lambda)-\trigamf(\para)}\\
&\le  \Vvv(\log( Z_n^{\para}(t)) +  C(\para) bu  \le   \Vvv(\log( Z_n^{\para}(t)) +  C(\para) u.
\end{align*}
We can drop $b$ from the upper bound above because the earlier requirements force
$b$ to be small. So we can safely assume  form the beginning that $b\le 1$.  

Continuing from (\ref{anotherUP}), 
\bea\nn
\P(Q^\para_{n,t}(\sigma_0^+\ge u)\ge e^{-su^2/n})
&\le& C \frac{n^2}{u^4}\Vvv(\log( Z_n^{\para}(t)) + C\frac{n^2}{u^3}. 
\eea
Using Theorem \ref{varthm} and once more $u\ge 3An^{2/3}$  we get
\bea\nn
\P(Q^\para_{n,t}(\sigma_0^+\ge u)\ge e^{-su^2/n}) 
&\le& C \frac{n^2}{u^4}E_{n,t} (\sigma_0^+) +CA n^{8/3} u^{-4} + C\frac{n^2}{u^3}\\ \nn 
&\le&  C \frac{n^2}{u^4}E_{n,t} (\sigma_0^+) +C \frac{n^2}{u^3}
\eea
which is exactly what we wanted to prove.

To proof for $\sigma^-$ starts with  $\lambda=\para-\frac{bu}{n}$ with $b>0$ small. 
Using  \be\nn \ind(s_0\le -u) e^{\para s_0}\le e^{-(\para-\lambda) u} \ind(s_0\le -u) e^{\lambda s_0}\ee
we get
\be\nn
Q_{n,t}^{\para}(\sigma_0^{-}\ge u)\le \frac{Z_{n}^{\lambda}(t)}{Z_n^{\para}(t)} e^{-(\para-\lambda)u}
\ee 
and the rest of the proof goes the same way.
\end{proof}

\begin{lemma}\label{4.1lem}
For  $\para,  \lambda>0$,   
\[
\bigl\lvert \Vvv(\log( Z_n^{\lambda}(t)) - \Vvv(\log( Z_n^{\para}(t)) \bigr\rvert 
\le  n\abs{\trigamf(\lambda)-\trigamf(\para)}.  
\]
\end{lemma}
\begin{proof}
Assume $\lambda>\para$.  From Theorem \ref{varthm}
\begin{align*}
\Vvv(\log( Z_n^{\lambda}(t))-\Vvv(\log( Z_n^{\para}(t))&=n(\trigamf(\lambda)-\trigamf(\para))
+2 \bigl[ E^\lambda_{n,t} (\sigma_0^+) -  E^\para_{n,t} (\sigma_0^+)\bigr]\\
&=-n(\trigamf(\lambda)-\trigamf(\para))+2 \bigl[ E^\lambda_{n,t} (\sigma_0^-) - 
 E^\para_{n,t} (\sigma_0^-)\bigr].
\end{align*} 
From the definition (\ref{Qdef2}) of  the quenched expectation  
\begin{align*}
\frac\partial{\partial\para}E^{Q^{\para}} (\sigma_0^{\pm})= \Cov^{Q^\para}(\sigma_0,\sigma_0^{\pm}). 
\end{align*}
For any random variable 
\begin{align*}
\Cov(X,X^{\pm})&=E ((X^+-X^-)X^\pm)-E (X^+-X^-)E X^\pm \\
&=\pm\Var (X^\pm) \pm E X^+ E X^-. 
\end{align*}
Consequently 
\[
 E^\lambda_{n,t} (\sigma_0^-) -  E^\para_{n,t} (\sigma_0^-)\le 0 \le
E^\lambda_{n,t} (\sigma_0^+) -  E^\para_{n,t} (\sigma_0^+).   
\]
The claim follows. 
\end{proof}
Next the tail bound for $\sigma_0^\pm$ for larger deviations.  

\begin{lemma}   Assume \eqref{assA}.  Let $\delta>0$.  
 Then there exist  $c=c(\para, \delta)<\infty$ and 
  $s=s(\para, \delta)>0$ such that for 
  $u\ge \max\{\delta n, 3An^{2/3}\} $
\be
\P(Q_{n,t}(\sigma_0^\pm\ge u)\ge e^{-su})\le 2 e^{-c u}.
\ee
\label{ubldlemma}\end{lemma}
\begin{proof}
We do the case $\sigma_0^+$.  The argument for $\sigma_0^-$ 
is analogous. Set $\lambda=\para+\nu$ where $\nu>0$ is small.
Then
\bea 
\nn
\P(Q^\para_{n,t}(\sigma_0^+\ge u)\ge e^{-su})&\le& \P(\frac{Z_n^{\lambda}(t)}{Z_{n}^\para(t) }e^{(\para-\lambda) u}\ge e^{-su})\\
&=&\P(\log( Z_n^{\lambda}(t))-\log( Z_n^{\para}(t))\ge (\lambda-\para)u-su).\nn
\eea 
After centering:
\bea \nn
&&\P(Q^\para_{n,t}(\sigma_0^+\ge u)\ge e^{-su})\le\\
&&\hskip15mm
\P\bigl(\overline{\log( Z_n^{\lambda}(t))}-\overline{\log( Z_n^{\para}(t))}\ge 
n(\digamf(\lambda)-\digamf(\para))-{t}(\lambda-\para)+(\lambda-\para)u-su\bigr).\label{ldp1}
\eea 
From a Taylor expansion, $u\ge  \max\{\delta n, 4An^{2/3}\} $, and then fixing $\nu$   small
enough and $s=\nu/2$,   
\begin{align*}
&n(\digamf(\lambda)-\digamf(\para))- {t}(\lambda-\para)+(\lambda-\para)u-su \\
&\qquad \ge 
-C(\para) \nu^2 n -A \nu n^{2/3} +\nu u-s u\\
&\qquad\ge C(\delta, \para) u.  \nn
\end{align*} 
From (\ref{Zincr1}) and (\ref{Z0}) 
\be \label{diff}
\overline{\log( Z_n^{\lambda}(t))}-\overline{\log( Z_n^{\para}(t))}=\sum_{j=1}^n \overline{r_j^{\lambda}(t)}-\sum_{j=1}^n \overline{r_j^{\para}(t)}  
\ee
where, for a fixed $t$,  
 $e^{-r^\para_j(t)}$  are i.i.d.~$\textup{Gamma}(\para,1)$ variables, and similarly $e^{-r^\lambda_j(t)}$with parameter $\lambda$.
 Thus for  certain sums $S'_n, S''_n$ of i.i.d.\ mean zero variables with an exponential moment
\[ 
\P(Q^\para_{n,t}(\sigma_0^+\ge u)\ge e^{-su})\le  \P( S'_n-S''_n\ge Cu) \le {\P(S'_n\ge Cu/2)+\P(S''_n<-Cu/2)\le} e^{-cu} 
 \]
where the last inequality follows from $u\ge\delta n$ and   standard large deviation theory.
\end{proof}

Now we combine the two deviation estimates into a  moment bound.  

\begin{lemma}   Assume \eqref{assA}.   
 Then there exists  $C(\para)<\infty$ such that 
\be  E_{n,t}(\sigma_0^\pm) \le (4A+C(\para))n^{2/3}.
\label{sigmamom}\ee
\end{lemma}
\begin{proof}  We do the computation for $\sigma_0^+$. It is identical for $\sigma_0^-$.
  Let $r\ge 1$, to be chosen at the end,
and $B=r\vee (3A)$.  Begin with 
\begin{align}
 E_{n,t}(\sigma_0^+) \le Bn^{2/3} + \int_{Bn^{2/3}}^{n\vee (Bn^{2/3})} P(\sigma_0^+\ge u)\,du 
 + \int_{n\vee (Bn^{2/3})}^\infty P(\sigma_0^+\ge u)\,du.
\label{int3}\end{align}
The last integral in \eqref{int3} is bounded by a constant that depends on $\para$,
 uniformly over $r>0$, $A>0$ and  $n\ge 1$,  
as can be seen by an application of  Lemma \ref{ubldlemma} with $\delta=1$. 
The middle integral is bounded as follows, utilizing Lemma \ref{lem_sigtail} with $K=1$.  
\begin{align*}
  &\int_{Bn^{2/3}}^{n\vee (Bn^{2/3})} P(\sigma_0^+\ge u)\,du \\
 &\le   \int_{Bn^{2/3}}^{n\vee (Bn^{2/3})} \left\{ e^{-s u^2/n}+\P(Q_{n,t}(\sigma_0^+\ge u) \ge e^{-su^2/n})\right\}\,du
\\
   &\le \int_{Bn^{2/3}}^\infty   \biggl(   \frac{C(\para) n^2}{u^4} 
E_{n,t} (\sigma_0^+) + \frac{C(\para) n^2}{u^3}  \biggr) \,du   + C(\para) \\[3pt]
&\le  \frac{C(\para)}{B^3} E_{n,t} (\sigma_0^+) + \frac{C(\para) }{B^2}n^{2/3}  + C(\para).   \end{align*}
Combine the estimates, noting that $B\ge 1$ and $n\ge 1$,  to 
\[    E_{n,t}(\sigma_0^+) \le (B+C(\para)) n^{2/3} + \frac{C(\para)}{r^3} E_{n,t} (\sigma_0^+). \]
Choose $r=(4C(\para))^{1/3}$.  Rearranging gives the conclusion.  
\end{proof}

Theorem \ref{UBthm} is now proved by \eqref{sigmamom}, the variance identity \eqref{var1} 
and assumption \eqref{assA}.   For future reference let us also state tail bounds
on $\sigma_0^\pm$ that we obtain by combining \eqref{sigmamom} with   Lemmas
 \ref{lem_sigtail} and \ref{ubldlemma}.
 
 \begin{proposition}
Under assumption \eqref{assA} we have these tail bounds, for finite positive  constants 
$C$, $c$ and  $s$ that depend on $\para$, and for $b>0$:    
\be  \label{sigmatail1}
\P(Q_{n,t}({\sigma_0}^\pm \ge b n^{2/3})\ge  e^{-s b^2 n^{1/3}})\le C b^{-3} 
\quad\text{for $3A\le b\le n^{1/3}$,}
\ee
\be  \label{sigmatail2}
\P(Q_{n,t}({\sigma_0}^\pm \ge b n^{2/3})\ge  e^{-s b n^{2/3}})\le 2e^{-cbn^{2/3}}
\quad\text{for $b\ge n^{1/3}\vee(3A)$}
\ee
and 
\be 
\label{sigmatail}
P_{n,t}(\sigma_0^\pm \ge b n^{2/3})\le C b^{-3}\quad\text{for $b\ge 3A$}.
\ee 
\label{sigmaprop}\end{proposition} 
From \eqref{sigmatail} we get the moment bound 
\be   E_{n,t}(\,\abs{\sigma_0}^p\,) \le (3^pA^p+C(\para))n^{2p/3}\quad
\text{for $1\le p<3$.}  
\label{sigmapmom}\ee


\section{Lower bound on the variance}
\begin{theorem}\label{LBthm}
Assume that   \eqref{assA} holds. Then there exists a constant 
$C_1=C_1(A,\para)$ such that
\be 
\Vvv (\log Z_n^{\para}(t))\ge C_1 n^{2/3}.
\ee
\end{theorem}
The   estimate that gives the theorem is in the next proposition.  

\begin{proposition} Assume that  \eqref{assA} holds with a constant
$A\in\bR_+$.
  Then there exist
finite positive $\para$-dependent 
constants $C(\para), c(\para), D(\para) $ so that,  if  $0<\delta\le 1$ and $K\ge 1$ 
satisfy \[  D(\para)(A+1)\delta^{1/2}\le K\le c(\para)(A+1)^{-4}\delta^{-1/2}, \] 
then 
\be \nn   \varlimsup_{n\to \infty} \P\bigl(Q^\para_{n,t}(0< \sigma_0\le \delta n^{2/3})>e^{-K n^{1/3} \sqrt{\delta}}\,\bigr)\le C(\para)(e^{-K^2/16}+ K^{3/4} \delta ^{3/8}).
\ee
\label{LBprop1}\end{proposition}
\begin{remark}
As a corollary we   get the following more general statement. Fix $x\in\bR$ and assume
that 
 \be  D(\para)(A+\abs{x}+1)\delta^{1/2}\le K\le c(\para)(A+\abs{x}+1)^{-4}\delta^{-1/2}. 
 \label{LBass9}\ee
Then
\be  \varlimsup_{n\to \infty} \P(Q^\para_{n,t}(x n^{2/3}< \sigma_0\le (x+\delta) n^{2/3})>e^{-K n^{1/3} \sqrt{\delta}})\le C(e^{-K^2/16}+ K^{3/4} \delta ^{3/8}).
\ee
This follows because by the translation invariance   \eqref{altQ2}
\be\nn 
Q^\para_{n,t}(x n^{2/3}< \sigma_0\le (\delta+x) n^{2/3})\eqd Q^\para_{n,t-x n^{2/3}}(0< \sigma_0\le \delta n^{2/3})
\ee 
and  $\abs{t-xn^{2/3}-n \trigamf(\para)}\le (A+\abs{x}) n^{2/3}$.
In particular, with $x=-\delta$, we get the  matching estimate for $\sigma_0^-$, and we can
combine the estimates for $\sigma^\pm_0$:  under assumptions \eqref{assA} and \eqref{LBass9} with $x=-\delta$,
\be    \varlimsup_{n\to \infty} \P\bigl(Q^\para_{n,t}(\abs{\sigma_0}\le \delta n^{2/3})>2e^{-K n^{1/3} \sqrt{\delta}}\,\bigr)\le C(\para)(e^{-K^2/16}+ K^{3/4} \delta ^{3/8}).
\label{LBprop1bound3}\ee
\end{remark}

Before proving 
  Proposition \ref{LBprop1}  let us observe how   Theorem \ref{LBthm} is proved.
Estimate \eqref{LBprop1bound3} gives the annealed limit 
\[
 \lim_{\delta \searrow 0} \varlimsup_{n\to \infty} P^\para_{n,t}(\abs{\sigma_0}\le \delta n^{2/3})=0.
\]
Then by the  variance identity  \eqref{var1}  
\[
\Vvv(\log Z_n^{\para}(t))
=E_{n,t}^\para\bigl(\abs{\sigma_0}\bigr) \ge \delta n^{2/3} P^\para_{n,t}(\abs{\sigma_0}\ge \delta n^{2/3}) \]
and  Theorem \ref{LBthm} follows.  

\begin{proof}[Proof of Proposition \ref{LBprop1}]
Set $u=\delta n^{2/3}, \upsilon(\delta)=K \sqrt{\delta}$ and begin by writing
\begin{align}
&\P(Q^\para_{n,t}(0< \sigma_0\le \delta n^{2/3})>e^{-n^{1/3} \upsilon(\delta)})=\P\left(\frac{Q^\para( \sigma_0>u\textup{ or } \sigma_0<0)}{Q^\para( 0< \sigma_0 \le u)}<e^{n^{1/3} \upsilon(\delta)}-1\right)\\
&\qquad\qquad 
=\P\left(\frac{Z^\para_{n,t}( \sigma_0>u \textup{ or } \sigma_0<0)}{Z^\para_{n,t}( 0< \sigma_0 \le u)}<e^{n^{1/3} \upsilon(\delta)}-1\right)\nn \\
&\qquad\qquad 
=\P\left(\frac{Z^\para_{n,t}( \sigma_0>u\textup{ or } \sigma_0<0)}{Z^\para_{n,t}( 0< \sigma_0 \le u)}\cdot\frac{Z_{1,n}(0,t)}{Z_{1,n}(0,t)}<e^{n^{1/3} \upsilon(\delta)}-1\right)\nn .
\end{align}
Split the last probability to get
\begin{align}
&\P(Q^\para_{n,t}(0< \sigma_0\le\delta n^{2/3})>e^{-n^{1/3} \upsilon(\delta)}) 
\ \le \ \P\left(\frac{Z^\para_{n,t}( \sigma_0>u\textup{ or } \sigma_0<0)}{Z_{1,n}(0,t)}<e^{2 n^{1/3} \upsilon(\delta)}\right)\nn\\[4pt]
&\qquad\qquad\qquad \qquad\qquad 
 +\P\left(\frac{Z^\para_{n,t}( 0< \sigma_0\le u)}{Z_{1,n}(0,t)}> e^{n^{1/3}  \upsilon(\delta)} \frac1{1-e^{-n^{1/3} \upsilon(\delta)}}\right)\nn\\[4pt]
&\qquad\qquad
 \le \ \P\left(\frac{Z^\para_{n,t}( \sigma_0>u)}{Z_{1,n}(0,t)}<e^{2 n^{1/3} \upsilon(\delta)}\right)
\label{LB1}\\[4pt]
&\qquad\qquad \qquad \qquad\qquad
+ \ \P\left(\frac{Z^\para_{n,t}( 0< \sigma_0\le u)}{Z_{1,n}(0,t)}> e^{n^{1/3}  \upsilon(\delta)}\right). \label{LB2}
\end{align}
We   bound    probabilities in (\ref{LB1}) and (\ref{LB2})  separately.

\medskip
\noindent \textbf{The term (\ref{LB1}).} 
\be \nn 
\frac{Z^\para_{n,t}( \sigma_0>u)}{Z_{1,n}(0,t)}=\int_u^t \exp(-B(s)+\para s) \frac{Z_{1,n}(s,t)}{Z_{1,n}(0,t)} ds.
\ee 
Construct a new environment $\tilde\om$  with 
\be\nn 
\tilde B(s)=-(B_n(t)-B_n(t-s)), \quad \tilde B_i(s)=B_{n-i}(t)-B_{n-i}(t-s), \quad 1\le i \le n-1, 
\ee
and take a new parameter $\lambda=\para+a(\delta) n^{-1/3}$ where $a(\delta)=K^{-1/4} \delta^{-1/8}$.
Quantities  that use environment $\tilde\om$ are marked with a tilde. 
From the definitions one checks that 
\be
Z_{1,n}(s,t)=\tilde Z_{0,n-1}(0,t-s) \quad\text{ for any $t>0$ and  $s\in(-\infty,t)$.} 
\label{ZZtil}\ee
Use  (\ref{comp2}) for the new system to get
\bea \nn 
\frac{Z_{1,n}(s,t)}{Z_{1,n}(0,t)}&=&\frac{\tilde Z_{0,n-1}(0,t-s)}{\tilde Z_{0,n-1}(0,t)}\ge \frac{\tZ_{n-1,t-s}^{\lambda}(\sigma_0>0)}{\tZ_{n-1,t}^{\lambda}(\sigma_0>0)}=\frac{\tilde Q^\lambda_{n-1,t-s}(\sigma_0>0) \tZ_{n-1,t-s}^{\lambda}}{\tilde Q^\lambda_{n-1,t}(\sigma_0>0)\tZ_{n-1,t}^{\lambda}}\\
\nn &\ge&\tilde Q^\lambda_{n-1,t-s}(\sigma_0>0)\frac{ \tZ_{n-1,t-s}^{\lambda}}{\tZ_{n-1,t}^{\lambda}}=
\tilde Q^\lambda_{n-1,t-s}(\sigma_0>0) \exp(\tilde Y_{n-1}(t-s,t)-\lambda s).
\eea 
Thus, denoting   the probability in \eqref{LB1}  by $p_1$, 
\begin{align*}
p_1&=\P\left(\frac{Z^\para_{n,t}( \sigma_0>u)}{Z_{1,n}(0,t)}<e^{2 n^{1/3} \upsilon(\delta)}\right)\\ 
&\le 
\P\left(
\int_u^t  e^{-B(s)+\tilde Y_{n-1}(t-s,t)+(\para-\lambda) s}
\,\tilde Q^\lambda_{n-1,t-s}(\sigma_0>0) \,ds<e^{2 n^{1/3} \upsilon(\delta)}
\right) \\
&=  \P\left(
\int_u^t  e^{-B(s)+\tilde Y^*_{n-1}(t-s,t)+(\para-\lambda) s}
\,\tilde Q^{\lambda,*}_{n-1}(t-s-\sigma^*_{n-1}>0) \,ds<e^{2 n^{1/3} \upsilon(\delta)}
\right).   
\end{align*}
On the last line above we applied the $*$ transformation to the $\tilde \omega$ system and Lemma \ref{LEMrev} 
to replace the measure $\tilde Q^{\lambda, \tilde \om^*}$ with the dual measure $\tilde Q^{\lambda,*}$.  

 Set 
\[
\bar u=n\trigamf(\para)-(n-1)\trigamf(\lambda). 
\]
Given any $\delta$ and $K$,  there exists a constant $c_0(\para)>0$ such that 
   $\bar u/2\ge c_0(\para)a(\delta)n^{2/3}$  for large enough $n$. 
Furthermore, from the hypothesis on    $\delta$ and $K$ it follows that 
  $\bar u/2\ge u$ and $\bar u/2\ge 3An^{2/3}$. 
Restrict the integration inside the probability and decompose  again:   
\begin{align}
p_1&\le \P\left(
\int_u^{\bar u/2}   e^{-B(s)+\tilde Y^*_{n-1}(t-s,t)+(\para-\lambda) s}
\,\tilde Q^{\lambda,*}_{n-1}(t-\sigma^*_{n-1}>\bar u/2) \,ds<e^{2 n^{1/3} \upsilon(\delta)}
\right)\nn\\
&\label{eq:lower0}\le \P\left(\tilde Q^{\lambda,*}_{n-1}(t-\sigma^*_{n-1}>\bar u/2)\le 1/2\right)\\
&\label{eq:lower1} \qquad +\;
\P\left(
\int_u^{\bar u/2} \exp(-B(s)+\tilde Y^*_{n-1}(t-s,t)+(\para-\lambda) s) ds<2e^{2 n^{1/3} \upsilon(\delta)}
\right). 
\end{align}
For probability (\ref{eq:lower0}) switch to complements, apply Lemma \ref{LEMrev} again,
and then  the upper bound (\ref{sigmatail1}):
\be\begin{aligned}
 &\P\left(\tilde Q^{\lambda,*}_{n-1}(t-\sigma^*_{n-1}>\bar u/2)\le  1/2\right)
 = \P\left(\tilde Q^{\lambda,*}_{n-1}(t-\bar u-\sigma^*_{n-1}\le -\bar u/2)> 1/2\right)\\
&\qquad   =\P\left( \tilde Q^\lambda_{n-1,t-\bar u}(\sigma_0 \le -\bar u/2) >1/2\right)
 \le C(\para) a(\delta)^{-3}=C(\para)K^{3/4}\delta^{3/8}. \end{aligned}\label{eq:lower2}\ee
On the last line above assumption \eqref{assA} continues to be valid with the 
same constant $A$ because 
$
\abs{t-\bar u -(n-1)\trigamf(\lambda)}=\abs{t-n\trigamf(\para)}\le An^{2/3}.  
$
Property $\bar u/2\ge 3An^{2/3}$ is the assumption needed for (\ref{sigmatail1}).  
 


Next we  estimate probability (\ref{eq:lower1}). 
Process $s\mapsto \tilde Y^*_{n-1}(t-s,t)$  is 
  a standard Brownian motion, and   independent of $B$ because
$\tilde Y^*$ was constructed from the new environment $\tilde\om$.   Define another 
standard Brownian motion  
\be\nonumber B^\dagger (s)= 2^{-1/2}n^{-1/3}  \bigl(-B(n^{2/3}s)+\tilde Y^*_{n-1}(t-n^{2/3}s,t)\bigr).\ee 
Then  probability \eqref{eq:lower1} equals  
\begin{align}
&\P\left(
\int_u^{\bar u/2} \exp( \sqrt{2}n^{1/3}B^\dagger(n^{-2/3}s)+(\para-\lambda) s) \,ds<2e^{2 n^{1/3} \upsilon(\delta)}
\right)\nn\\ 
&\qquad \le 
\P\left(n^{2/3}
\int_\delta^{c_0(\para)a(\delta)} \exp(\sqrt{2} n^{1/3} B^{\dagger}(s)-a(\delta) n^{1/3} s) \,ds<2e^{2 n^{1/3} \upsilon(\delta)}
\right)\nn\\
&\qquad =
\P\left(n^{-1/3} \log
\int_\delta^{c_0(\para)a(\delta)}  e^{\sqrt{2} n^{1/3} B^{\dagger}(s)-a(\delta) n^{1/3} s}
\,ds\le2 \upsilon(\delta) +n^{-1/3} \log(2n^{-2/3})
\right)  \nn\label{eq:lower2}
\end{align}
 As $n\to\infty$ the probability on the last line above converges to
\[
\P\left(\sup_{\delta\le s\le c_0(\para)a(\delta)}(
\sqrt{2} B^{\dagger}(s)-a(\delta)  s) \le2 \upsilon(\delta) 
\right).
\]
 Introduce one more   Brownian motion  
$B(s)= B^{\dagger}(\delta+s)-B^{\dagger}(\delta)$. 
Abbreviate temporarily  $\tau=a(\delta)^2 (c_0(\para)a(\delta)-\delta)/\sqrt 2 >1$
where the inequality is a consequence of the assumption on $\delta$ and $K$.  
 Then the 
probability above is  
\begin{align*}
&\le \P\bigl(\sqrt{2} B^{\dagger}(\delta)<- \upsilon(\delta)\bigr)
+\P\left(\sup_{0\le s\le c_0(\para)a(\delta)-\delta}(
\sqrt{2} B(s)-a(\delta)  s) \le 3 \upsilon(\delta)+\delta a(\delta) 
\right)\nn\\ \nn
&\le  e^{- \frac14\upsilon(\delta)^2 \delta^{-1}}
+\P\left(\sup_{0\le s\le  \tau}(
 B(s)-  s) \le 2^{-1/2}a(\delta)(3 \upsilon(\delta)+\delta a(\delta) )
\right)\\
&\le  e^{- \frac14\upsilon(\delta)^2 \delta^{-1}}
+   \P\left(\sup_{0\le s\le 1}(
 B(s)-  s) \le 2^{-1/2}a(\delta)(3 \upsilon(\delta)+\delta a(\delta))
\right)\\[3pt]
&\le e^{- \frac14\upsilon(\delta)^2 \delta^{-1}}
+  C a(\delta)(3\upsilon(\delta)+\delta a(\delta) ). 
\end{align*}
The last inequality comes because  $\sup_{0\le s\le 1} ( B(s)-  s)$ is a.s.~positive with a bounded density function. Including the estimate from \eqref{eq:lower2}  we get
\be \label{eq:LB1}
\varlimsup_{n\to \infty} \, (\ref{LB1})\le  C(\para) (e^{- K^2/4  }+ K^{3/4}\delta^{3/8} 
 +  K^{-1/2}\delta^{3/4}  ) \le   C(\para)(e^{- K^2/4  }+ K^{3/4}\delta^{3/8}) .
\ee 
{In the last inequality we used $\delta\le 1\le K$.}

\medskip
\noindent \textbf{The term (\ref{LB2}).} 

For  probability  (\ref{LB2}) we separate the argument into a lemma because the
same estimate will be needed again, though with different parameters. 

\begin{lemma} Assume \eqref{assA} with the constant $A$ and let $\aaa,\bbb, \kappa>0$.  
Then there exist finite, positive  
constants $C(\para)$, $C_1(\para)$ and  $n_0(a,b,\kappa,\para)$ such that, if $b/a\ge C_1(\para)(A+1)$, then 
for $n\ge n_0(a,b, \kappa, \para)$
\[ \P\left(\frac{Z^\para_{n,t}( 0< \sigma^\pm_0\le \aaa n^{2/3})}{Z_{1,n}(0,t)}\ge \kappa e^{n^{1/3} \bbb}\right)
\le  
C(\para) \bigl( a^3b^{-3} + \exp(-b^2a^{-1}/16) \bigr).  
  \]
\label{LBlemma3}\end{lemma}


Before proving the lemma let us use it to conclude the proof of Proposition \ref{LBprop1}. 
In Lemma \ref{LBlemma3}  take $\kappa=1$, $\aaa=\delta$ and $\bbb=K\sqrt\delta$.
 Then for large enough $n$, 
\be 
 \text{probability (\ref{LB2}) } \le C(\para) (\delta^{3/2}K^{-3}+ e^{-K^2/16}).
\label{eq:LB2}\ee

Combine (\ref{eq:LB1}) and  (\ref{eq:LB2})  with $\delta\le 1\le K$, and we have
\bea\nn 
 \varlimsup_{n\to \infty} \P(Q^\para_{n,t}(0< \sigma_0\le \delta n^{2/3})>e^{-n^{1/3} \upsilon(\delta)}) 
  \le   C(\para) ( e^{-K^2/16}+  K^{3/4} \delta ^{3/8}   )
\nn
\eea 
and  the proposition is proved.
 \end{proof}

\begin{proof}[Proof of Lemma \ref{LBlemma3}]
  We do the case of $\sigma^+_0$ in full detail.  Abbreviate $u=\aaa n^{2/3}$.  
Introduce the new environment $\tilde \omega$ as before, and a new parameter
  $\lambda=\para-r  n^{-1/3}$ with $r =b/(4a)$.  
 We must restrict $n$ large enough so that for example $rn^{-1/3}<\para/2$ so that
 $\lambda$ is a legitimate parameter.   
  
Begin with \eqref{ZZtil}   and then apply   comparison  \eqref{comp2}: 
\begin{align*}
\frac{Z_{1,n}(s,t)}{Z_{1,n}(0,t)}& =\frac{\tilde Z_{0,n-1}(0,t-s)}{\tilde Z_{0,n-1}(0,t)} \le \frac{\tZ_{n-1,t-s}^{\lambda}(\sigma_0<0)}{\tZ_{n-1,t}^{\lambda}(\sigma_0<0)}=
\frac{\tZ_{n-1,t-s}^{\lambda}}{\tZ_{n-1,t}^{\lambda}}\cdot \frac{\tilde Q^{\lambda}_{n-1,t-s}(\sigma_0<0)}{\tilde Q^{\lambda}_{n-1,t}(\sigma_0<0)}\\
&\le  \exp\bigl(\tilde Y_{n-1}(t-s,t)-\lambda s\bigr)\cdot \frac{1}{\tilde Q^{\lambda}_{n-1,t}(\sigma_0<0)}. 
\end{align*}
Substitute the above bound in the probability that is to be bounded:   
\bea \nn&&
\P\left(\frac{Z^\para_{n,t}( 0< \sigma_0\le u)}{Z_{1,n}(0,t)}\ge\kappa e^{n^{1/3}\bbb }\right)\\
&&\hskip10mm \nn =\P\left(\int_0^u \exp(-B(s)+\para s) \frac{Z_{1,n}(s,t)}{Z_{1,n}(0,t)}\, ds
\ge \kappa  e^{n^{1/3}  b}  \right)\\
&&\hskip10mm \nn \le
\P\left(\int_0^u \frac{ \exp(-B(s)+\tilde Y_{n-1}(t-s,t)+(\para-\lambda) s)  }{\tilde Q^{\lambda}_{n-1,t}(\sigma_0<0)}\,ds\ge\kappa e^{n^{1/3}  \bbb}\right)\\[3pt] 
&&\hskip10mm
\le\P\left(\tilde Q^{\lambda}_{n-1,t}(\sigma_0<0)\le 1/2\right)\label{eq:lower3.9} \\
&&\hskip20mm
+ \ \P\left(\int_0^u \exp(-B(s)+\tilde Y_{n-1}(t-s,t)+(\para-\lambda) s) \, ds\ge \frac\kappa{2}  e^{n^{1/3}  \bbb}\right).\label{eq:lower4}
\eea 

To treat probability (\ref{eq:lower3.9}) set 
$\bar u=  (n-1)\trigamf(\lambda)-n\trigamf(\para)$.  
$\trigamf$ is positive, convex and strictly decreasing, so one can check that 
 $\bar u \ge\tfrac14 \abs{\trigamf'(\para)} r n^{2/3}$ for  all $n\ge 1$ provided $C_1(\para)$ in the hypothesis
is large enough.   Use the shift 
invariance property of $Q$ described in Remark \ref{remark1} and the upper bound (\ref{sigmatail1}):
\be\begin{aligned}
&\P(\tilde Q^{\lambda}_{n-1,t}(\sigma_0<0)<1/2)=\P(\tilde Q^{\lambda}_{n-1,t}(\sigma_0>0)\ge1/2)\\
&\qquad =
\P(\tilde Q^{\lambda}_{n-1,t+\bar u}(\sigma_0>\bar u)\ge1/2) \le C(\para) r^{-3}
\le C(\para)(a/b)^3. 
 \end{aligned}\label{auxlemma8}\ee
The choice of $\bar u$ makes \eqref{assA} valid again with the same $A$, and  
a large enough $C_1(\para)$  guarantees that $\bar u\ge 3An^{2/3}$ so that 
(\ref{sigmatail1}) can be applied.  
 
For probability (\ref{eq:lower4}), after rescaling the integral and introducing a new Brownian motion, 
\begin{align*}
&\text{(\ref{eq:lower4})}  \le \P\left(n^{2/3}\int_0^\aaa \exp(n^{1/3}(\sqrt{2}B^{\dagger}(s)+r  s))  ds>\frac{\kappa}2  e^{n^{1/3}  \bbb}\right)\\   
&  =\P\left(n^{-1/3}\log\int_0^\aaa \exp(n^{1/3}(\sqrt{2}B^{\dagger}(s)+r  s))  ds> \bbb +n^{-1/3} \log( \kappa n^{-2/3}/2)\right)\\  
&\le  
\P\left(\sup_{0\le s\le \aaa}(
\sqrt{2} B^{\dagger}(s)+r   s) \ge \tfrac34\bbb
\right).
\end{align*}
In the last inequality we took $n$ large enough so that $n^{-1/3} \log( \kappa n^{-2/3}/2)< b/4$.
Via $\sup_{0\le s\le \aaa}
 B^{\dagger}(s)\eqd {\aaa}^{1/2} \abs{B^{\dagger}(1)}$  bound the last probability by
\[
\P\left(\sup_{0\le s\le \aaa}
\sqrt{2} B^{\dagger}(s)\ge \tfrac34b-r  \aaa
\right)\le C\exp\left( -\,\frac1{4\aaa}{( \,\tfrac34b-r  \aaa)^2}\right)
=C\exp\left(-\,\frac{b^2}{16\aaa}\right). 
\]
Combining estimate \eqref{auxlemma8} with above gives the conclusion for $\sigma^+_0$.  

The case of $\sigma^-_0$ goes similarly, with small alterations.
Now $\lambda=\para+rn^{-1/3}$.  Utilizing 
\eqref{ZZtil} and comparison \eqref{comp2}
the ratio is  developed as follows:
\begin{align*}
&\frac{Z^\para_{n,t}( -u\le  \sigma_0<0)}{Z_{1,n}(0,t)} 
=
 \int_{-u}^0 \exp(-B(s)+\para s) \frac{Z_{1,n}(s,t)}{Z_{1,n}(0,t)}\, ds\\
 &\qquad \le  \int_{-u}^0 \frac{ \exp(-B(s)-\tilde Y_{n-1}(t,t-s)-(\para-\lambda) s)  }{\tilde Q^{\lambda}_{n-1,t}(\sigma_0>0)}\,ds.
\end{align*}
The rest follows along the same lines as above. With this we consider
Lemma \ref{LBlemma3}   proved.  
\end{proof}

\section{Fluctuations of the path under boundary conditions}

\begin{theorem}
Assume \eqref{assA} holds,  let $0<\gamma<1$ and assume $b\ge 3(A+1)$.  Then for
$n\ge (1-\gamma)^{-1}$
\be \label{path1}
P(\abs{\sigma_{\lfloor \gamma n \rfloor }-\gamma t}>b n^{2/3})\le C(\para) b^{-3}.
\ee 
Also, for any $0<\gamma<1, \varepsilon>0$ there exists $\delta>0$ with 
\be\label{path2} 
\varlimsup_{n\to \infty}P(|\sigma_{\lfloor \gamma n \rfloor }-\gamma t|\le \delta n^{2/3})\le \varepsilon.
\ee 
\end{theorem}

\begin{proof}
For the first statement it is enough to prove that 
\be \label{eq:path1} Q_{n,t}(\sigma_k-v>u)\eqd Q_{n-k,t -v }(\sigma_0>u).\ee 
Indeed, from this identity we get
\be
P_{n,t}(|\sigma_{\lfloor\gamma n\rfloor}-\gamma t |>b n^{2/3})=P_{n-\lfloor\gamma n\rfloor, (1-\gamma)t} (|\sigma_0|>b n^{2/3})\le C(\para) b^{-3}
\ee
where the last inequality comes from  applying (\ref{sigmatail}).  This is legitimate
because 
\[\abs{(1-\gamma)t-(n-\lfloor\gamma n\rfloor)\trigamf(\para)}\le (A+1)n^{2/3}.\]  
Condition $n\ge (1-\gamma)^{-1}$ ensures that $n-\lfloor\gamma n\rfloor\ge 1$.

By Lemma \ref{LEMrev}, to prove (\ref{eq:path1})   it is enough to show that the  distribution of $(\sigma_1^*,\dots, \sigma_{n-1}^*)$ is the same under $Q^*_n$ and $Q^*_{n-1}$.  For this   
 we check that integrating out $\sigma_{n}^*$ from the density function of $Q^*_n$ results 
in the density of $Q^*_{n-1}$.  In the next calculation use \eqref{rXY} and \eqref{defYk}. 
\begin{align*}
&\int_{s_{n-1}}^{\infty}\frac1{Z_n^*} \exp\left[ X_1(0,s_1)+X_2(s_1,s_2)+\dotsm+ X_n(s_{n-1},s_{n})+Y_n(0,s_{n})-\para s_{n}\right] ds_{n}\\ 
&  =\frac1{Z_n^*} \exp\left[ X_1(0,s_1) +\dotsm+ X_{n-1} (s_{n-2},s_{n-1})+Y_n(0,s_{n-1})-\para s_{n-1}\right] \\ \nn
& \qquad \times
\int_{s_{n-1}}^\infty  \exp\left[X_{n}(s_{n-1},s_n)+Y_n(s_{n-1}, s_n)-\para (s_n-s_{n-1}) \right]ds_{n}\\
& =
\frac1{Z_n^*} \exp\left[ X_1(0,s_1) +\dotsm+ X_{n-1} (s_{n-2},s_{n-1})+Y_n(0,s_{n-1})-\para s_{n-1}\right] e^{r_n(s_{n-1})}\nn
\\
& =\frac{e^{r_n(0)}}{Z_n^*} \exp\left[ X_1(0,s_1) +\dotsm+ X_{n-1} (s_{n-2},s_{n-1})+Y_{n-1}(0,s_{n-1})-\para s_{n-1}\right] \nn
\end{align*}
which is exactly the density of $Q^*_{n-1}$ and also shows that $Z_n^*=Z_{n-1}^* e^{r_{n}(0)}$.

 To prove   (\ref{path2}), use  (\ref{eq:path1}) to write 
\be\nn 
Q_{n,t}(|\sigma_{\lfloor \gamma n\rfloor} -\gamma t|\le \delta n^{2/3})\eqd Q_{n-\lfloor \gamma n\rfloor,(1-\gamma)t}(|\sigma_{0} |\le \delta n^{2/3})
\ee 
and apply Proposition \ref{LBprop1}.
 \end{proof}

\section{Upper bounds   without boundary conditions}
\label{sec:freeZ}

\begin{theorem}\label{thm:freeZ}  Let $\tau>0$ and pick $\para$ so that
$\trigamf(\para)=\tau$. 
Then for $n\ge n_0(\tau)$ and $b\ge b_0(\tau)$  we have 
\be \label{free0}
\P(\abs{\log Z_{1,n}(0,n\tau)-n(\trigamf(\para)\para- \digamf(\para)) }
\ge b n^{1/3})\le C(\tau) b^{-3/2}.
\ee 
 \end{theorem}
\begin{proof}
The choice of $\para$ gives $\bE(\log  Z_ {n,n\tau}^{\para})=n(\trigamf(\para)\para- \digamf(\para))$, so by 
 Theorem \ref{UBthm} we only need to prove the bound 
\be
\P(|\log Z_{1,n}(0,n\tau)-\log Z_ {n,n\tau}^{\para}|\ge b n^{1/3})
\le C b^{-3/2}.
\ee
Abbreviate $t=n\tau$.  
By (\ref{ZB2}) 
\be\label{bound}
Z_{n,t}^\para\ge e^{r_1(0)} Z_{1,n}(0,t).
\ee
This gives 
\bea 
\P\left( \log Z_ {n,t}^{\para} - \log Z_{1,n}(0,t) \le -b n^{1/3} \right)
\le \P\left(e^{r_1(0)}\le e^{-b n^{1/3}}  \right) \le C e^{-b n^{1/3}}, \nn   
\eea 
the last inequality follows from  $e^{r_1(0)}\sim \textup{Gamma}(\para,1)^{-1}$ which has bounded density   near 0. 
 
To get the opposite bound set $u=\sqrt{b} n^{2/3}$  and write
\be\begin{aligned}
&\P\left(\frac{Z_{n,t}^\para}{Z_{1,n}(0,t)}\ge e^{b n^{1/3}}\right)=\P\left(\frac{Z_{n,t}^\para(|\sigma_0|\le u)}{Z_{1,n}(0,t)\, Q_{n,t}^\para(|\sigma_0|\le u)}\ge e^{b n^{1/3}}   \right)\\
& \qquad \qquad \le \P\left( \frac{Z_{n,t}^\para(|\sigma_0|\le u)}{Z_{1,n}(0,t) }\ge \frac12 e^{b n^{1/3}}   \right)+
\P\left(Q_{n,t}^\para(|\sigma_0|\le u)\le 1/2\right)\\
&\qquad \qquad \le  C(\para)b^{-3/2}. 
\end{aligned} \label{free3}\ee
To get the last inequality, apply Lemma \ref{LBlemma3} with $\aaa=\sqrt b$
to the first probability,  
and the upper bound  (\ref{sigmatail1}) to the second probability, 
and take both $n$ and $b$ large
enough. 
\end{proof}

Theorem \ref{zeta2thm}  is a restatement of Theorem \ref{thm:freeZ}. The next theorem proves Theorem \ref{thm:freepath2}.

\begin{theorem}\label{thm:freepath}
Assume \eqref{assA} holds,  
 and  $0<\gamma<1$. Then for large enough $n$
 and $b$    we have
\be 
P_{(1,n),(0,t)}\left(|\sigma_{\lfloor n \gamma\rfloor}-\gamma t|>b n^{2/3}   \right)\le C(\para) b^{-3}.
\ee 
\end{theorem}
\begin{proof}
Let $\ell =\lfloor n \gamma\rfloor, t'=\gamma t$ and $u=b n^{2/3}$. By the definitions and   (\ref{bound})  
\begin{align*}
&Q_{(1,n),(0,t)}\left(| \sigma_{\ell }-t'|>u  \right)=
\frac1{Z_{1,n}(0,t)}   \int\limits_{\abs{s-t'}>u}  {Z_{1,\ell }(0,s) Z_{\ell +1,n}(s,t)} \,ds\\
 &\quad \le \frac{e^{-r_1(0)}}{Z_{1,n}(0,t)}   \int\limits_{\abs{s-t'}>u}  {Z^\para_{\ell ,s}\, Z_{\ell +1,n}(s,t)} \,ds  \; = \; 
 \frac{e^{-r_1(0)}Z_{n,t}^\para}{Z_{1,n}(0,t)} \,Q_{n,t}\left(| \sigma_{\ell }-t'|>u  \right) .
\end{align*} 
Consider  $h\in(b^{-3},1)$.  
\bea \nn 
\P\left(Q_{(1,n),(0,t)}\left(| \sigma_{\ell }-t'|>u  \right)>h  \right)&\le& \P(e^{r_1(0)}\le b^{-3})+\P\left[
\frac{Z_{n,t}^\para}{Z_{1,n}(0,t)} \ge e^{r n^{1/3}}
\right]\\
\nn&&+\P\left[  
Q_{n,t}\left(| \sigma_{\ell }-t'|>u  \right) >e^{-r n^{1/3}} h b^{-3}
\right]
\eea 
where we set $r=s b^2/(3(1-\gamma))$ with $s$ from Proposition \ref{sigmaprop}. 
The first term is bounded by $C b^{-3}$ as $e^{r_1(0)}$ has bounded density near zero. The second term is bounded by $C r^{-3/2}\le C b^{-3}$ by (\ref{free3}). Finally, \eqref{eq:path1} and Lemma \ref{lem_sigtail} give, for large enough $n$ and $b$ and uniformly for $h\in(b^{-3},1)$,  
\begin{align*}
&\P\left[   Q_{n,t}\left(| \sigma_{\ell }-t'|>u  \right) >e^{-r n^{1/3}} h b^{-3}
\right]\le \P\left[  
Q_{n,t}\left(| \sigma_{\ell }-t'|>u  \right) >e^{-2r n^{1/3}}
\right]\\
 &\qquad\qquad =\P\left[Q_{n-\ell ,t-t'}(|\sigma_0|>u)>e^{-s u^2/(n-\ell) }\right] 
\le   C b^{-3}.
\end{align*}
Collecting the estimates
\[
\P\left[Q_{(1,n),(0,t)}\left(| \sigma_{\ell }-t'|>u  \right)>h  \right]\le C  b^{-3}
\]
and from this 
\begin{align*}
P_{(1,n),(0,t)}\left(|\sigma_{\lfloor n \gamma\rfloor}-\gamma t|>b n^{2/3}   \right)&\le b^{-3}+\int_{b^{-3}}^1 \P\left[Q_{(1,n),(0,t)}\left(| \sigma_{\ell }-t'|>u  \right)>h  \right] dh\le C b^{-3}.
\end{align*}
This completes the proof of the theorem.
\end{proof}

\bibliography{growthrefs}

\bibliographystyle{abbrv}

\end{document}